\documentclass[11pt]{amsart}
\usepackage{amssymb,amsmath,epsfig,mathrsfs, enumerate, xparse}
\usepackage[nodisplayskipstretch]{setspace}
\usepackage{graphicx}
\usepackage[normalem]{ulem}
\usepackage{fancyhdr}
\usepackage{tikz}
\usepackage[margin=2.5cm]{geometry}
\usepackage{hyperref}
\hypersetup{
 colorlinks   = true,
 urlcolor     = blue,
 linkcolor    = blue,
 citecolor   = red ,
 bookmarksopen=true
}

\theoremstyle{plain}
\newtheorem{thrm}{Theorem}[section]
\newtheorem{lemma}[thrm]{Lemma}

\newtheorem{rmrk}[thrm]{Remark}
\newtheorem{dfn}[thrm]{Definition}

\allowdisplaybreaks
\begin{document}
\newcommand{\psr}{P^-}
\newcommand{\qr}{Q_\rho}
\newcommand{\sn}{\mathbb{S}^{n-1}}
\newcommand{\SL}{\mathcal L^{1,p}( D)}
\newcommand{\Lp}{L^p( Dega)}
\newcommand{\CO}{C^\infty_0( \Omega)}
\newcommand{\Rn}{\mathbb R^n}
\newcommand{\Rm}{\mathbb R^m}
\newcommand{\R}{\mathbb R}
\newcommand{\Om}{\Omega}
\newcommand{\Hn}{\mathbb H^n}
\newcommand{\aB}{\alpha B}
\newcommand{\eps}{\ve}
\newcommand{\BVX}{BV_X(\Omega)}
\newcommand{\p}{\partial}
\newcommand{\IO}{\int_\Omega}
\newcommand{\bG}{\boldsymbol{G}}
\newcommand{\bg}{\mathfrak g}
\newcommand{\bz}{\mathfrak z}
\newcommand{\bv}{\mathfrak v}
\newcommand{\Bux}{\mbox{Box}}
\newcommand{\e}{\ve}
\newcommand{\X}{\mathcal X}
\newcommand{\Y}{\mathcal Y}
\newcommand{\W}{\mathcal W}
\newcommand{\la}{\lambda}
\newcommand{\vf}{\varphi}
\newcommand{\rhh}{|\nabla_H \rho|}
\newcommand{\Ba}{\mathcal{B}_\beta}
\newcommand{\Za}{Z_\beta}
\newcommand{\ra}{\rho_\beta}
\newcommand{\na}{\nabla_\beta}
\newcommand{\vt}{\vartheta}

\numberwithin{equation}{section}

\newcommand{\RN} {\mathbb{R}^N}
\newcommand{\Sob}{S^{1,p}(\Omega)}
\newcommand{\Dxk}{\frac{\partial}{\partial x_k}}
\newcommand{\Co}{C^\infty_0(\Omega)}
\newcommand{\Je}{J_\ve}
\newcommand{\beq}{\begin{equation}}
\newcommand{\bea}[1]{\begin{array}{#1} }
\newcommand{\eeq}{ \end{equation}}
\newcommand{\ea}{ \end{array}}
\newcommand{\eh}{\ve h}
\newcommand{\Dxi}{\frac{\partial}{\partial x_{i}}}
\newcommand{\Dyi}{\frac{\partial}{\partial y_{i}}}
\newcommand{\Dt}{\frac{\partial}{\partial t}}
\newcommand{\aBa}{(\alpha+1)B}
\newcommand{\GF}{\psi^{1+\frac{1}{2\alpha}}}
\newcommand{\GS}{\psi^{\frac12}}
\newcommand{\HFF}{\frac{\psi}{\rho}}
\newcommand{\HSS}{\frac{\psi}{\rho}}
\newcommand{\HFS}{\rho\psi^{\frac12-\frac{1}{2\alpha}}}
\newcommand{\HSF}{\frac{\psi^{\frac32+\frac{1}{2\alpha}}}{\rho}}
\newcommand{\AF}{\rho}
\newcommand{\AR}{\rho{\psi}^{\frac{1}{2}+\frac{1}{2\alpha}}}
\newcommand{\PF}{\alpha\frac{\psi}{|x|}}
\newcommand{\PS}{\alpha\frac{\psi}{\rho}}
\newcommand{\ds}{\displaystyle}
\newcommand{\Zt}{{\mathcal Z}^{t}}
\newcommand{\XPSI}{2\alpha\psi \begin{pmatrix} \frac{x}{\left< x \right>^2}\\ 0 \end{pmatrix} - 2\alpha\frac{{\psi}^2}{\rho^2}\begin{pmatrix} x \\ (\alpha +1)|x|^{-\alpha}y \end{pmatrix}}
\newcommand{\ZZ}{ \begin{pmatrix} xx^{t} & (\alpha + 1)|x|^{-\alpha}x y^{t}\\
     (\alpha + 1)|x|^{-\alpha}x^{t} y &   (\alpha + 1)^2  |x|^{-2\alpha}yy^{t}\end{pmatrix}}
\newcommand{\norm}[1]{\lVert#1 \rVert}
\newcommand{\ve}{\varepsilon}
\newcommand{\D}{\operatorname{div}}

\title[Parabolic Harnack etc.]{An Intrinsic  Harnack inequality for some non-homogeneous  parabolic equations in non-divergence form}

\author{Vedansh Arya*}
\address{Tata Institute of Fundamental Research\\
Centre For Applicable Mathematics \\ Bangalore-560065, India}\email[Vedansh Arya]{vedansh@tifrbng.res.in}


%
\thanks{*vedansh@tifrbng.res.in}
\thanks{Tata Institute of Fundamental Research, Centre For Applicable Mathematics, Bangalore-560065, India}



%
%
%

\maketitle
\begin{abstract}

In this paper, we establish a scale invariant Harnack inequality for some inhomogeneous parabolic equations  in a suitable intrinsic geometry   dictated by the nonlinearity.  The class of equations that we consider correspond to the parabolic counterpart  of the   equations studied by Julin in \cite{j} where a generalized Harnack inequality was obtained which quantifies the strong maximum principle.  Our version of parabolic Harnack  (see Theorem \ref{mthm})  when restricted to the elliptic case   is  however quite different from that in \cite{j}. The key new feature   of this work is an appropriate modification of the  stack of  cubes covering  argument  which is   tailored for the nonlinearity that we consider.

\end{abstract}

\tableofcontents
 
\section{Introduction and the statement of the main result}

In this paper we consider parabolic equations of the type
\begin{align}\label{maineq}
    F(D^2u, Du, x, t)-u_t=0,
\end{align}
where  $F$ is assumed to be uniformly elliptic in the Hessian and has  a certain nonlinear growth in the gradient variable. More precisely,  we assume that  there exist constants $0 < \lambda \leq \Lambda$ such that 
\begin{align*}
    \lambda Tr(N) \leq F(M+N,p,x,t)-F(M,p,x,t) \leq \Lambda Tr(N)
\end{align*}
$\forall\  N \geq 0$  and for every $(p,x,t) \in \R^n \times Q_2$. We furthermore assume that $F$ has the following growth dependence in the gradient variable,
\begin{align}\label{nonlin}
    |F(0,p,x,t)| \leq \phi(|p|)
\end{align}
for every $(p,x,t) \in \R^n \times Q_2$  and where  $\phi : [0,\infty) \rightarrow [0,\infty)$ is of the form $\phi(t)=\eta(t)t$ and satisfies the following conditions analogous to that in \cite{j}:
\begin{enumerate}
    \item[(P1)] $\phi : [0,\infty) \rightarrow [0,\infty)$ is increasing, locally Lipschitz continuous in $(0,\infty)$ and $\phi(t)\geq t$ for every $t\geq0$. Moreover, $\eta : [0,\infty) \rightarrow [1,\infty)$ is nonincreasing on $(0,1)$ and nondecreasing on $[1,\infty]$;
    \item[(P2)] $\eta$ satisfies 
    \begin{equation*}
        \underset{t\rightarrow\infty}{\text{lim}}\frac{t\eta^{'}(t)}{\eta(t)}\text{log}(\eta(t))=0;
    \end{equation*}
    \item[(P3)] There is a constant $\Lambda_0$ such that 
    \begin{align*}
        \eta(st)\leq \Lambda_0\eta(s)\eta(t);
    \end{align*}
    for every $s$, $t\in (0,\infty)$.

\end{enumerate}

Following the ideas of Caffarelli \cite{Ca},  we will replace the equation (\ref{maineq}) by two extremal  inequalities which takes into account the   ellipticity assumption and the  growth condition on the  drift term. In other words,  we assume that $u \in C(Q_2)$ is a viscosity supersolution of
\begin{equation}\label{eq1}
    P^-(D^2u)-u_t\leq\phi(|Du|)
\end{equation}
and  a viscosity subsolution of 
\begin{equation}\label{eq2}
    P^+(D^2u)-u_t\geq-\phi(|Du|),
\end{equation}
where $P^{\pm}$  correspond to the extremal Pucci operators as defined in \eqref{pucci}. We refer to Section \ref{s:n} for the precise notion of viscosity sub/supersolutions.  

Before proceeding further, we make the following discursive remark.

\begin{rmrk}
Throughout this paper, by a universal constant, we refer to a constant $C$ which  depends only  on the  ellipticity constants,  the nonlinearity $\phi$ and the  dimension $n$.\end{rmrk}

\subsection*{Statement of the main result}
We now state the main result of the paper which is regarding  the validity of a parabolic Harnack type inequality  for solutions to \eqref{maineq} in a  suitable intrinsic geometry    corresponding to the nonlinearity.

\begin{thrm}\label{mthm}
Let $u \in C(Q_2)$ be a positive viscosity supersolution of (\ref{eq1}) and viscosity subsolution of (\ref{eq2}). There is a universal constant $C>0$ such that 
\begin{align}\label{Harnack}
    \underset{A}{\text{sup}}\hspace{0.8mm} u(a_0 x, a_0^2 t) \leq  C u(0,0) \hspace{2mm}\text{for} \hspace{2mm} a_0=\frac{u(0,0)}{C(\phi(u(0,0))+u(0,0))},
\end{align}
where $A=\Big\{(x,t):|x|_{\infty} \leq \frac{c_n}{2}, -1+ \frac{c_n^2}{4} \leq t \leq -1+\frac{c_n^2}{2}\Big\}$. Here $c_n$ depends only on $n$.
\end{thrm}

\begin{rmrk}
Theorem \ref{mthm} roughly  ensures that  in a cube of size  $a_0 \sim \frac{u(0,0)}{\phi(u(0,0))+u(0,0)}$, the  parabolic Harnack inequality   holds. It is to be noted that when $\phi(t) \equiv t$, we have that $a_0 \approx 1$ and  consequently \eqref{Harnack} reduces to the scale invariant Krylov-Safonov type parabolic Harnack inequality as in \cite{KS}.  We would also like to  mention that in view of a counterexample in \cite{j}, the standard  Harnack inequality doesn't hold for solutions to \eqref{maineq}.

\end{rmrk}

\begin{rmrk}
In section \ref{s:n}, we provide an explicit example of a function satisfying the extremal inequalities in  \eqref{eq1} -\eqref{eq2}  above  which vanishes in finite time. Such an example would thus demonstrate  that the time lag in the Harnack inequality  for positive solutions to \eqref{eq1}-\eqref{eq2}   has to  depend on the solution in general. \end{rmrk}

 Now in  order to provide a proper perspective to our work, we  mention that   in the time independent case,  Julin in \cite{j}  showed that   nonnegative   viscosity solutions to
 \begin{equation}
 \begin{cases}
 P^-(D^2u)\leq\phi(|Du|) 
 \\
 P^+(D^2u)\geq-\phi(|Du|),
 \end{cases}
 \end{equation} 
in $B_2$  satisfy the following generalized Harnack inequality,
 \begin{equation}\label{eHarnack}
 \int_{\inf_{B_1} u}^{\sup_{B_1} u} \frac{dt}{ \phi(t) + t} \leq C.
 \end{equation}
 Such a Harnack inequality as in \eqref{eHarnack}  above quantifies the strong maximum principle.   Moreover  as mentioned above, a  counterexample in \cite{j} shows that the standard Harnack inequality cannot hold in this setting. The proof of the Harnack inequality in \cite{j} involves a  fairly delicate  adaptation of the Vitali type covering argument which utilizes the slow growth assumption on $\eta$ in a very crucial way.   It is to be noted that the classical proof of the Krylov-Safonov type weak Harnack estimate    as  in \cite{Ca}   for elliptic equations and \cite{W1} for the parabolic equations is  based on a  basic measure estimate which uses the Alexandrov-Bakeman-Pucci (ABP) type maximum principle following which  an   iterative argument is set up   on the super level sets of supersolutions  using the   Calderon-Zygmund decomposition.   In the parabolic case, an additional  ``stack of cubes" covering argument  is required which takes into account the time lag between the measure and the pointwise information in the basic measure estimate.  The switch from the Calderon-Zygmund type decomposition to the more elementary Vitali  argument is relatively well known in the elliptic setting ( see for instance \cite{IS1}, \cite{Sa}).  However in the parabolic case, because of such a  time lag in the measure estimate,  adapting the  Vitali type covering which takes into account the parabolic geometry  appears like a serious obstruction. Thus  it is not clear to us  as to whether  the methods in \cite{j} can be generalized to the parabolic case  in order to obtain an analogous parabolic Harnack inequality  as   \eqref{eHarnack} above. Therefore in this present work, we obtain a  new intrinsic framework where a  parabolic Harnack inequality can be established for \eqref{maineq}. Moreover in the course of the proof, we also show that in such a framework, the weak Harnack estimate for positive supersolutions also  remains  valid. This makes our work somewhat different from  \cite{j} where instead  a conditional weak Harnack estimate  is proven ( see Lemma  4.8 in \cite{j}) which however is optimal in the extrinsic Euclidean geometry and  suffices for the Harnack estimate \eqref{eHarnack}.
 
\medskip
 
We  proceed as in  the classical case as   in \cite{W1},  but  similar to that in \cite{j},  the   difficulty arises  because of the fact that if $u$ solves (\ref{eq1}) and $M>0$,  then 
\begin{align*}
    v(x,t)=\frac{u(x,t)}{M}
\end{align*}
is a solution of $$P^-(D^2u)-u_t\leq \Lambda_0\eta(M){\phi}(|Du|),$$ and thus the growth of the nonlinearity gets altered. However after further rescaling, 
\begin{align*}
    v(x,t)=u(rx,r^2t)
\end{align*}
 we observe  that  $v$ solves $$P^-(D^2v)-v_t\leq \Lambda_0 r \eta(1/r){\phi}(|Dv|).$$ At this point, by noting that   $r \eta(1/r) \to 0$ as $r \to 0$, we obtain that in a small enough region dictated by the nonlinearity, the growth conditions on the nonlinear term remains comparable to the initial assumptions   and thus one can normalize and use the equation in such  a region. 
 
Subsequently, the basic measure estimate is obtained via the ABP type comparison principle which relies on a somewhat  subtle computation of an appropriate  barrier  function ( see the proof of Lemma \ref{barrier}). The $L^{\epsilon}$-estimate is then obtained  by  applying  the measure estimate repeatedly on  appropriate normalized solutions using a   delicate stack of cubes covering argument. Over here, we would like to emphasize that  the stack of cubes defined in \cite{W1} ( see also \cite{is})  doesn't  work in our inhomogeneous situation. Therefore, we define  a different stack of cubes which is  tailor-made  for our nonlinearity. This constitutes the key  novelty of this work. Moreover, in our entire analysis, the  precise growing nature of $\eta$  crucially comes into play at various steps. Following the   $L^{\epsilon}$-estimate, the  subsolution estimate is then obtained by adapting the ideas in \cite{j} to our parabolic situation. We refer to the subsquent work  \cite{AJ} where a related boundary Harnack inequality in $C^{1,1}$ domains has been established for such inhomogeneous elliptic equations. 

\medskip

In closing, we  would like to mention that  various types of intrinsic Harnack inequalities for divergence  form parabolic equations modelled on 
\[
\operatorname{div}(|Du|^{p-2} Du) = u_t,
\]
have been obtained  in a  series of   fundamental works by DiBenedetto, Gianazza and Vespri (see  \cite{DGV1, DGV2, DGV3}) which in part has also inspired our present work. See also \cite{Ku} where an intrinsic weak Harnack inequality has been obtained for such equations.    It remains to be seen whether our techniques can be extended to such structures. We also refer to  a  recent interesting work \cite{PV} where intrinsic Harnack inequality for 
\[
|Du|^{\gamma} \Delta_p u= u_t,\ \gamma> 1-p,
\]
has been established  by comparison with explicit  Barenblatt solutions which in turn is inspired by an approach  of  DiBenedetto for the parabolic p-Laplacian. ( see \cite{Di}). 

  We would also  like to mention some other works \cite{IS1, Sa, Wa}  which introduces new techniques
for proving regularity estimates for nondivergence form elliptic and parabolic equations. These works study
equations which are homogeneous but which fail to be uniformly elliptic. The key idea is to
bypass the classical ABP estimate by directly touching the solution from below by paraboloids (or
by other suitable functions).\\

The paper is organized as follows. In Section \ref{s:n}, we introduce some basic notations and notions  and gather some preliminary results.  In Section \ref{s:m}, we prove our main result. Finally in the appendix, we provide proofs of Lemma \ref{barrier} and Lemma \ref{measureuniform} which involves a somewhat delicate and long computation.\\
\\
\textbf{Acknowledgments}

We would like to  thank  Agnid Banerjee for various helpful discussions and suggestions. We would also like to thank the editor for the kind handling of the  paper and also the reviewer for various valuable comments and suggestions  which has substantially  improved the presentation of this article. Especially, we are  grateful to the  reviewer for suggesting the possibility of finding an example of a solution which vanishes in finite time  which we have now put in section \ref{s:n}. Such an example   demonstrates that the intrinsic nature of  our Harnack inequality in Theorem \ref{mthm} is unavoidable.

\section{Notations and Preliminaries}\label{s:n}
A point in space time will be denoted by $(x,t)$ where $(x,t) \in \R^n \times \R.$
We will denote a point in space by $x$, $y$ etc.. We will use $0$ for both origin of $\R^n$ and real number. We will denote Euclidean norm of $x$ in $\R^n$ by $|x|$.
$Df$ and $D_{x,t}f$ will denote the gradient of $f$ in $x$ variable and gradient of $f$ in $x$ and $t$ both variables respectively. $f_t$ will denote partial derivative of $f$ with respect to $t$. $D^2f$ will represent the Hessian matrix of $f$ with respect to $x$.
For a given set $A$, $|A|$ will denote the Lebesgue measure of $A$.

A cube of radius $\rho$ and center $(x_0,t_0)$ is defined as following
\begin{align*}
    \qr(x_0,t_0) = \{x \in \R^n:|x-x_0|_{\infty} < \rho \}\times (t_0-\rho^2,t_0),
\end{align*}
where $|x|_{\infty} = \text{max}\{ |x_1|, |x_2|,..., |x_n|\}$.

We define
\begin{align}\label{tildq}
    \tilde {Q}_{\rho}(x_0,t_0) := \{(x,t) \hspace{1mm}: \hspace{1mm}|x-x_0|_{\infty} < 3{\rho},\hspace{1mm} t_0 <t< t_0+9{\rho}^2\}.
\end{align}
We denote $Q_{\rho}(0,0)$ by $Q_{\rho}$ and $\tilde {Q}_{\rho}(0,0)$ by $\tilde {Q}_{\rho}$.

For $(\Omega\times (t_1,t_2)) \subset \R^n \times \R$, 
$\partial_p (\Omega\times (t_1,t_2))$ will denote the parabolic boundary of  $\Omega \times (t_1,t_2)$  and defined as:
\begin{align*}
\partial_p(\Omega\times (t_1,t_2))&=\Omega\times \{t_1\}\cup\partial\Omega\times[t_1,t_2].
\end{align*}

Let $S$ be the space of real $n\times n$ symmetric matrices. We recall the definition of Pucci's extremal operators (for more detail see \cite{cc}). For $M\in S$, Pucci's extremal operators with ellipticity constant $0<\lambda\leq\Lambda$ are defined as
\begin{align}\label{pucci}
    P^-(M,\lambda,\Lambda)=P^-(M)=\lambda\sum_{e_i>0}e_i+\Lambda\sum_{e_i<0}e_i,
    \\
    P^+(M,\lambda,\Lambda)=P^+(M)=\Lambda\sum_{e_i>0}e_i+\lambda\sum_{e_i<0}e_i\notag,
\end{align}
where $e_i$'s are the eigenvalues of $M$.

We recall the definition of a viscosity supersolution of (\ref{eq1}) a viscosity subsolution of (\ref{eq2}).
\begin{dfn}
A function $u :Q_{\rho} \rightarrow \R$ is a viscosity supersolution of (\ref{eq1}) in $Q_\rho$ if it is continuous and the following holds: if $(x_0,t_0)\in Q_\rho$ and $\varphi \in C^2(Q_\rho)$ is such that $\varphi\leq u$ and $\varphi(x_0,t_0)=u(x_0,t_0)$ then
\begin{align*}
    P^-(D^2\varphi(x_0,t_0))-\varphi_t(x_0,t_0)\leq\phi(|D\varphi(x_0,t_0)|). 
\end{align*}
\end{dfn}
\begin{dfn}
A function $u :Q_{\rho} \rightarrow \R$ is a viscosity subsolution of (\ref{eq2}) in $Q_\rho$ if it is continuous and the following holds: if $(x_0,t_0)\in Q_\rho$ and $\varphi \in C^2(Q_\rho)$ is such that $\varphi\geq u$ and $\varphi(x_0,t_0)=u(x_0,t_0)$ then
\begin{align*}
    P^+(D^2\varphi(x_0,t_0))-\varphi_t(x_0,t_0)\geq-\phi(|D\varphi(x_0,t_0)|). 
\end{align*}
\end{dfn}

We will state some more properties of $\eta$. See Proposition 2.3 in \cite{j}.
\begin{lemma}\label{etaprop}
Let $\eta$ satisfy the conditions $(P1)-(P3)$. Then the following hold:
\begin{itemize}
    \item[(i)] For every $c>0$,  we have 
    \begin{equation*}
        \underset{t\rightarrow\infty}{\text{lim}}\frac{\eta(ct)}{\eta(t)}=1.
    \end{equation*}
    \item[(ii)] For every $\gamma>0$, we have
    \begin{equation*}
        \underset{t\rightarrow\infty}{\text{lim}}\frac{\eta(t)}{t^\gamma}=0.
    \end{equation*}
    \item[(iii)] There is a constant $\Lambda_1$ such that for every $t>0$ it holds that
    \begin{align*}
        \eta(\eta(t)t)\leq \Lambda_1 \eta(t).
    \end{align*}
    \item[(iv)] There is a constant $\Lambda_2$ such that for every $t>0$ and $0<r<s$ it holds that
    \begin{align*}
        r\eta(t/r)\leq \Lambda_2s\eta(t/s).
    \end{align*}
\end{itemize}

\end{lemma}

Since the equation (\ref{eq1}) is not scaling invariant, we need the following lemma. Proof is same as in the elliptic case, see Lemma $4.4$ in \cite{j}.
\begin{lemma}\label{scaling}
Let $u\in C(Q_2)$ be a viscosity supersolution of (\ref{eq1}) in $Q_2$. There exists a universal constant $L_2$ such that if $A\in (0, \infty)$ then for every $r\leq r_A$, where
\begin{align*}
    r_A=\frac{A}{L_2(\phi(A)+A)}=\frac{1}{L_2(\eta(A)+1)}
\end{align*}
the rescaled function 
\begin{align*}
    \tilde{u}(x,t):=\frac{u(rx,r^2t)}{A},
\end{align*}
is a supersolution of (\ref{eq1}) in its domain.
 \end{lemma}
 \begin{dfn}\label{medef}
\textbf{(Monotone envelope of a function)}. The monotone envelope of a lower semi-continuous function $u : Q_\rho(x,t) \rightarrow \R$ is the largest function $v : Q_\rho(x,t) \rightarrow \R$ lying below $u$ which is convex with respect to $x$ and non-increasing with respect to $t$. It is denoted by $\Gamma_u.$
\end{dfn}
We need the following lemma to prove Lemma \ref{measureuniform}. See Lemma $4.4$ in \cite{is}.
\begin{lemma}\label{glem}
If $u$ is $C^{1,1}$ with respect to $x$ and Lipschitz continuous with respect to $t$,
then the function $G : \Omega \times (a,b) \rightarrow \R^{n+1}$ defined as follows:
\begin{align*}
  Gu(x,t)=G(u)(x,t)=(Du(x,t),u(x,t)-x\cdot Du(x,t)),
\end{align*}
is Lipschitz continuous in $(x,t)$ and for a.e $(x,t) \in  \Omega \times (a,b) $,
\begin{align*}
    \text{det}\hspace{1mm} D_{x,t}G(u) = u_t\hspace{1mm} \text{det}\hspace{1mm}D^2u.
\end{align*}
\end{lemma}
We need the following lemma to prove Lemma \ref{measureuniform}. See Lemma $4.13$ in \cite{is}. 
\begin{lemma}\label{setmeas}
If $u \in C(Q_1)$ with $u \geq 0$ on $\partial_p(Q_1)$ and $\text{sup}_{Q_1} u^- =1,$ then
\begin{align*}
     \Big\{(\xi,h) \in \R^n \times \R:|\xi| \leq \frac{1}{4}, \frac{5}{8} \leq -h \leq \frac{6}{8}\Big\} \subset {G\Gamma_u(Q_1 \cap C_u)}.
 \end{align*}
 where $C_u=\{u=\Gamma_u\}$ and $\Gamma_u$ represents the monotone envelope of min$\{u,0\}$ extended by $0$ to $Q_2$.
\end{lemma}
We will now introduce parabolic setting of Calderon-Zygmund decomposition. (See \cite{is})

Consider a cube $Q$ of the form $(x_0,t_0)+(-s,s)^n \times (0,s^2)$. A \textbf{dyadic cube} $K$ of $Q$ is obtained by repeating a finite number of times the following iterative process: $Q$ is divided into $2^{n+2}$ by considering all translations of the form $(0,s)^n \times (0,s^2/4)$ by vectors of the form $(ks,l(s^2/4))$ with $k\in \mathbb{Z}^n$ and $l \in \mathbb{Z}$ included in $Q$. $\bar{K}$ is called the \textbf{predecessor} of $K$ if $K$ is one of the $2^{n+2}$ cubes obtained from dividing $\bar{K}$.\\ 
In figure below cube $ABCD$ is the predecessor of $K$.
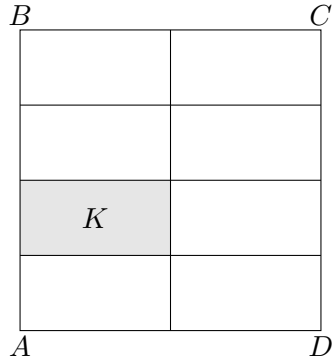
\begin{center}
\begin{tikzpicture}
	\fill[color=gray!20] (0,1) rectangle (2,2);
\draw (0,0) rectangle (4,4);
\draw[-] (0,1)--(4,1);
\draw[-] (0,2)--(4,2);
\draw[-] (0,3)--(4,3);
\draw[-] (2,0)--(2,4);
\draw (1,1.5) node{$K$};
\draw (0,-0.2) node{$A$};
\draw (4,-0.2) node{$D$};
\draw (0,4.2) node{$B$};
\draw (4,4.2) node{$C$};
\end{tikzpicture}
\end{center}

Let $m$ be a natural number. For a dyadic cube $K$ of $Q$, the set $\bar{K}^m$ is obtained by ``stacking" $m$ copies of its predecessor $\bar{K}$. More precisely, if $\bar{K}$ is of the form $\Omega \times (a,b)$ then $\bar{K}^m$ is $\Omega \times (b,b+m(b-a)).$
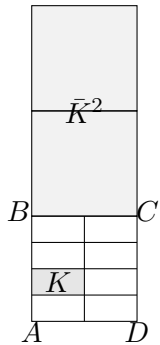
\begin{center}
\begin{tikzpicture}[scale=0.35]
	\fill[color=gray!10] (0,8) rectangle (4,12);
	\fill[color=gray!10] (0,4) rectangle (4,8);
	\fill[color=gray!20] (0,1) rectangle (2,2);
	\draw (0,0) rectangle (4,4);
\draw (0,4) rectangle (4,8);
\draw (0,8) rectangle (4,12);
\draw (2,8) node{$\bar{K}^2$};

\draw[-] (0,1)--(4,1);
\draw[-] (0,2)--(4,2);
\draw[-] (0,3)--(4,3);
\draw[-] (2,0)--(2,4);

\draw (1,1.5) node{$K$};
\draw (0,-0.4) node{$A$};
\draw (4,-0.4) node{$D$};
\draw (-0.5,4.2) node{$B$};
\draw (4.4,4.2) node{$C$};
\end{tikzpicture}
\end{center}

We need the following covering lemma. See Lemma $4.27$ in \cite{is}.
\begin{lemma}\label{cald}
Let $m$ be a natural number. Consider two subsets $A$ and $B$ of a cube $Q$. Assume that $|A|\leq \delta |Q|$ for some $\delta \in (0,1)$. Assume also the following: for any dyadic cube $K \subset Q$,
\begin{align*}
    |K \cap A|>\delta |K| \implies \Bar{K}^m \subset B.
\end{align*}
Then $|A| \leq \delta\frac{m+1}{m}|B|.$
\end{lemma}

  As mentioned in the introduction,  we now  provide an example of a function satisfying \eqref{eq1}-\eqref{eq2} which vanishes in finite time. 

\subsection*{Example.}
Consider for $(x,t) \in Q_2(0,1),$
\begin{align*}
	u(x,t)=
	\begin{cases}
		e^{1/t}(x+3) \hspace{2mm}&t<0\\
		0 & t\geq 0.
	\end{cases}
\end{align*}
Also let $\phi(s)=5s(|\operatorname{ln}(s)|+4)^2.$

Then for $t<0$,
\begin{align}\label{x1}
	 u_{xx}  -u_t=\frac{e^{1/t}(x+3)}{t^2} \geq 0 \geq -\phi(|u_x|)
\end{align}
and moreover, 
\begin{align}\label{x2}
		u_{xx} -u_t&=\frac{e^{1/t}(x+3)}{t^2}\\
		&=|u_x|(x+3)\left(|\operatorname{ln}(e^{1/t})|\right)^2\notag\\
		&\leq |u_x|(x+3)\left(|\operatorname{ln}(|u_x|)|+4\right)^2\notag\\
		&\leq 5|u_x|\left(|\operatorname{ln}(|u_x|)|+4\right)^2\notag\\
		&=\phi(|u_x|)\notag.
	\end{align}
Moreover since all the derivatives of $u$ decay as $t \to 0^{-}$, therefore \eqref{x1} and \eqref{x2} continues to remain valid for $t>0$. Thus $u(x,t)$ solves (\ref{eq1}) and (\ref{eq2}) corresponding to  $\phi(s)=5s(|\operatorname{ln}(s)|+4)^2.$

Now we show  that such a $\phi$ satisfies (P1), (P2) and (P3). Note that for $s<1,$
\begin{align*}
	\phi '(s)&=5(4-\operatorname{ln}s)^2 +10 s (4-\operatorname{ln}(s))(-1/s)=5(4-\operatorname{ln}s)(2-\operatorname{ln}s)>0,\\
	\eta '(s)&=\frac{-10}{s}(4-\operatorname{ln}s)<0.
\end{align*}
 For $s>1,$ we instead have, 
 \begin{align*}
 	\phi '(s)&=5(4+\operatorname{ln}s)(6+\operatorname{ln}s))>0,\\
 	\eta '(s)&=10(4+\operatorname{ln}s)(1/s)>0.
 \end{align*}
Also $\phi(s) \geq s.$ Thus $\phi$ satisfies (P1). Now (P2) is satisfied as
\begin{align*}
	\underset{s \rightarrow \infty}{\text{lim}}\frac{s \eta'(s)}{\eta(s)}\operatorname{ln}(\eta(s))=\underset{s \rightarrow \infty}{\text{lim}}\frac{10s (4+\operatorname{ln}s)(1/s)}{25(4+\operatorname{ln}s)^2}\operatorname{ln}((4+\operatorname{ln}s))= \underset{s \rightarrow \infty}{\text{lim}} \frac{10\operatorname{ln}(4+\operatorname{ln}s)}{25(4+\operatorname{ln}s)}=0.
\end{align*}
Observe that (P3) is also  satisfied as
\begin{align*}
	\eta(s_1)\eta(s_2)&=25(4+|\operatorname{ln}s_1|)^2(4+|\operatorname{ln}s_2|)^2\\
	&=25(16+4|\operatorname{ln}s_1|+4|\operatorname{ln}s_2|+|\operatorname{ln}s_1||\operatorname{ln}s_2|)^2\\
	& \geq 25(16+4|\operatorname{ln}s_1|+4|\operatorname{ln}s_2|)^2\\
	&  \geq 25(16+4|\operatorname{ln}s_1s_2|)^2= 80 \eta(s_1 s_2).                           
	\end{align*}

\section{Proof of Main Theorem}\label{s:m}


We will follow the same approach as in \cite{W1} but in the intrinsic setting. \\
First, we will construct a barrier function as Lemma $4.16$ in \cite{is}. Proof is in Appendix \ref{a}.

We define the following subsets of $Q_1$.
\begin{align*}
    {K_1}&:=(-c_n,c_n)^n \times (-1,-1+c_n^2),\\
    {K_2}&:=(-3c_n,3c_n)^n \times (c_n^2-1,10c_n^2-1),\\
    {K_3}&:=(-3c_n,3c_n)^n \times (-1+c_n^2,0),
\end{align*}
where $c_n=(10n)^{-1}.$

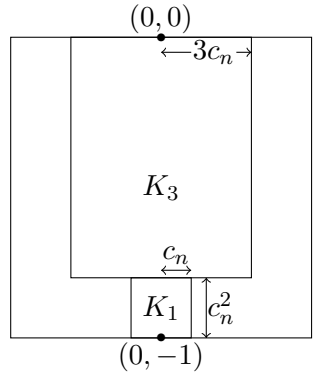
\begin{center}
    \begin{tikzpicture}{Scale=1.3}
    \draw (2,0) node{\tiny{$\bullet$}};
    \draw (2,-0.25) node{$(0,-1)$};
    \draw (0,0) rectangle (4,4);
    \draw (1.6,0) rectangle (2.4,0.8);
    \draw (0.8,0.8) rectangle (3.2,4);
    \draw (2,0.4) node{$K_1$};
    \draw (2,2) node{$K_3$};
    \draw (2,4) node{\tiny{$\bullet$}};
    \draw (2,4.25) node{$(0,0)$};
    \draw[<-] (2,3.8)--(2.4,3.8);
    \draw (2.7,3.8) node{$3c_n$};
    \draw[->] (3,3.8)--(3.2,3.8);
    \draw[<->] (2.6,0)--(2.6,0.8);
    \draw (2.8,0.4) node{$c_n^2$};
    \draw[<->] (2,0.9)--(2.4,0.9);
    \draw (2.2,1.1) node{$c_n$};
    \end{tikzpicture}
\end{center}
 
\begin{lemma}\label{barrier}
There exist a universal constant $r_0>0$ and a nonpositive  Lipschitz function $h:Q_{1}$ $\rightarrow$ $\R$, which is $C^2$ with respect to $x$ on the set where $h$ is negative and solves (in viscosity sense)
\begin{align*}
    P^+(D^2h)-h_t+\tilde{\phi}(|Dh|)\leq g
\end{align*}
for some continuous, bounded function $g:$ $Q_{1}$ $\rightarrow$ $\R$ and supp $g$ $\subset$ ${K_1}$,\\
where $\tilde{\phi} = {\Lambda_{0}}{r} \eta \big(\frac{1}{r}\big) \phi$ for $r \leq r_0$.
Also, $h\leq-2$ in ${K_3}$ and $h=0$ on $\partial_pQ_1$.
\end{lemma}
We will use previous barrier function to obtain the following basic measure estimate. 
\begin{lemma}\label{measureuniform}
There exist universal constants $L>1$, $\mu \in (0,1)$ and $r_1 \in (0,1)$ such that for any nonnegative supersolution of 
\begin{equation}\label{211}
    \psr(D^2u) - u_t \leq \tilde{\phi}(|Du|) \hspace{5mm} \text{in} \hspace{2mm}Q_1(0,0),
\end{equation}
where $\tilde{\phi} = {\Lambda_{0}}{r_1} \eta \big(\frac{1}{r_1}\big) \phi$,
the followings holds:\\
If ${\inf}_{K_3}u \leq 1$ then,
\begin{align*}
    |\{(x,t) \in K_1:u(x,t) \leq L\}| \geq \mu |K_1|.
\end{align*}

\end{lemma}

The proofs of both Lemma \ref{barrier} and Lemma \ref{measureuniform} are  given in the Appendix \ref{a} because they involve long computations.
\begin{rmrk}\label{rmu}
The above lemma is true for the case when $u$ is viscosity supersolution of (\ref{211}) in $ D=\{(x,t):|x|_{\infty}<1,-1<t\leq c_n^2+{\delta}^2\}$ for some $\delta >0$.\\
More precisely, if ${\inf}_{K_3 \cap D}u \leq 1$ then,
\begin{align*}
    |\{(x,t) \in K_1:u(x,t) \leq L\}| \geq \mu |K_1|.
\end{align*}
\end{rmrk}
\textbf{Stack of cubes.} As we mentioned in introduction stack of cubes defined in \cite{W1} ( also \cite{is}) does not work in our inhomogeneous situation. So, we will define stack of cubes which will work in our situation. First, we will explain it in general setting, later we will specify required quantities.\\
Let's say we are given a cube $\qr(x_0,t_0)$, for some $\rho \in (0,1)$ and $z_1$, $z_2$, . . ., $z_m$ such that $z_i(\leq 3)$. We will define stack of cubes, denoted by $Q^i(\qr(x_0,t_0))$ for $i=0,1,2,...,m$, corresponding to $\qr(x_0,t_0)$ and $z_1$, $z_2$, . . ., $z_m$. We will define $Q^i(\qr(x_0,t_0))$ inductively as follows:\\
For $i=0$,  define $Q^0(\qr(x_0,t_0)):=\qr(x_0,t_0)$.\\
Now,assume $Q^i(\qr(x_0,t_0))$ is defined for $i=k$. We will define a cube for $i=k+1$, i.e., $Q^{k+1}(\qr(x_0,t_0))$.
Let $Q^k(\qr(x_0,t_0))$ be a cube of radius $r_k$ and  center $(x_k,t_k)$, i.e., $Q^k(\qr(x_0,t_0))=Q_{r_k}(x_k,t_k)$.\\
Take a cube  with the following properties:
\begin{itemize}
    \item Radius of cube is $z_{k+1}{r_k}$.
    \item Center of cube is $(x_{k+1},t_k+(z_{k+1}{r_k})^2)$, where $x_{k+1}$ is such that $|x_{k+1}-x_k|_{\infty} \leq 3r_k-z_{k+1}{r_k} $.
    \item Cube is closest to the line $\{(0,t):t \in \R \}$.
\end{itemize}
This will be our $Q^{k+1}(\qr(x_0,t_0))$. Thus, we have defined our stack of cubes.

Let $Q^k(\qr(x_0,t_0))$ be a cube of radius $r_k$ and  center $(x_k,t_k)$, i.e., $Q^k(\qr(x_0,t_0))=Q_{r_k}(x_k,t_k)$. Then 
\begin{align}\label{tildqk}
	\tilde{Q}^k(\qr(x_0,t_0)):=\tilde{Q}_{r_k}(x_k,t_k),
\end{align}
where $\tilde{Q}_{r_k}(x_k,t_k)$ is defined in (\ref{tildq}).
In the figure below, $Q^k$ denotes $Q^{k}(\qr(x_0,t_0))$. The bigger cube represents $\tilde{Q}^k(\qr(x_0,t_0))$. For instance cube filled with dots is $\tilde{Q}^3$.
\noindent
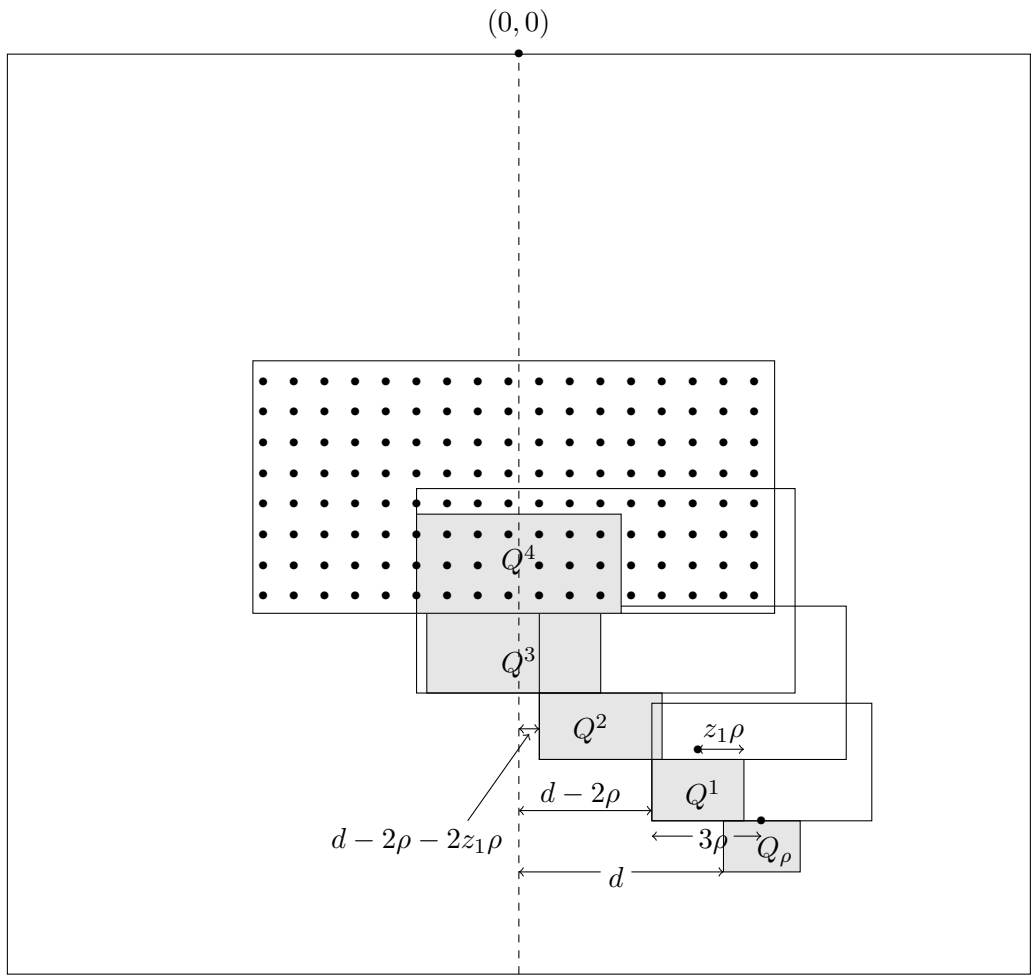
\begin{center}
\begin{tikzpicture}[scale=1.36]
	\fill[color=gray!20] (8,1) rectangle (8.75,1.5);
	\fill[color=gray!20] (7.3,1.5) rectangle (8.2,2.1);
	\fill[color=gray!20] (6.2,2.1) rectangle (7.4,2.75);
	\fill[color=gray!20] (5.1,2.75) rectangle (6.8,3.53);
	
\draw (01,0) rectangle (11,9);
\draw (8,1) rectangle (8.75,1.5);
\draw (8.5,1.2) node{$Q_{\rho}$};
\draw[<-] (6,1) -- (6.8,1);
\draw[->] (7.1,1)--(8,1);
\draw (6.95,0.95) node{$d$};
\draw (7.3,1.5) rectangle (9.45,2.65);
\draw (7.3,1.5) rectangle (8.2,2.1);
\draw[<->] (7.75,2.2)--(8.2,2.2);
\draw (7.75,2.2) node{\tiny{$\bullet$}};
\draw (8,2.35) node{$z_1\rho$};
\draw[<-] (7.3,1.35) -- (7.7,1.35);
\draw[->] (8.1,1.35) -- (8.37,1.35);
\draw (7.9,1.3) node{$3\rho$};
\draw (7.8,1.75) node{$Q^1$};
\draw (8.37,1.5) node{\tiny{$\bullet$}};
\draw[<->] (6,1.6) -- (7.3,1.6);
\draw (6.6,1.75) node{$d-2\rho$};
\draw (6.2,2.1) rectangle (9.2,3.6);
\draw (6.2,2.1) rectangle (7.4,2.75);
\draw[<->] (6,2.4)--(6.2,2.4);
\draw[->] (5.5,1.5)--(6.1,2.35);
\draw (5,1.3) node{$d-2\rho-2z_1\rho$};
\draw (6.7,2.4) node{$Q^2$};
\draw (5,2.75) rectangle (8.7,4.75);
\draw (5.1,2.75) rectangle (6.8,3.53);
\draw (6,3.05) node{$Q^3$};
\draw (3.4,3.53) rectangle (8.5,6);
\draw (3.5,3.7) node{\tiny{$\bullet$}};
\draw (3.8,3.7) node{\tiny{$\bullet$}};
\draw (4.1,3.7) node{\tiny{$\bullet$}};
\draw (4.4,3.7) node{\tiny{$\bullet$}};
\draw (4.7,3.7) node{\tiny{$\bullet$}};
\draw (7.7,3.7) node{\tiny{$\bullet$}};
\draw (8.0,3.7) node{\tiny{$\bullet$}};
\draw (8.3,3.7) node{\tiny{$\bullet$}};
\draw (3.5,4) node{\tiny{$\bullet$}};
\draw (3.8,4) node{\tiny{$\bullet$}};
\draw (4.1,4) node{\tiny{$\bullet$}};
\draw (4.4,4) node{\tiny{$\bullet$}};
\draw (4.7,4) node{\tiny{$\bullet$}};
\draw (7.7,4) node{\tiny{$\bullet$}};
\draw (8.0,4) node{\tiny{$\bullet$}};
\draw (8.3,4) node{\tiny{$\bullet$}};
\draw (3.5,4.3) node{\tiny{$\bullet$}};
\draw (3.8,4.3) node{\tiny{$\bullet$}};
\draw (4.1,4.3) node{\tiny{$\bullet$}};
\draw (4.4,4.3) node{\tiny{$\bullet$}};
\draw (4.7,4.3) node{\tiny{$\bullet$}};
\draw (7.7,4.3) node{\tiny{$\bullet$}};
\draw (8.0,4.3) node{\tiny{$\bullet$}};
\draw (8.3,4.3) node{\tiny{$\bullet$}};
\draw (3.5,4.6) node{\tiny{$\bullet$}};
\draw (3.8,4.6) node{\tiny{$\bullet$}};
\draw (4.1,4.6) node{\tiny{$\bullet$}};
\draw (4.4,4.6) node{\tiny{$\bullet$}};
\draw (4.7,4.6) node{\tiny{$\bullet$}};
\draw (7.7,4.6) node{\tiny{$\bullet$}};
\draw (8.0,4.6) node{\tiny{$\bullet$}};
\draw (8.3,4.6) node{\tiny{$\bullet$}};
\draw (3.5,4.9) node{\tiny{$\bullet$}};
\draw (3.8,4.9) node{\tiny{$\bullet$}};
\draw (4.1,4.9) node{\tiny{$\bullet$}};
\draw (4.4,4.9) node{\tiny{$\bullet$}};
\draw (4.7,4.9) node{\tiny{$\bullet$}};
\draw (7.7,4.9) node{\tiny{$\bullet$}};
\draw (8.0,4.9) node{\tiny{$\bullet$}};
\draw (8.3,4.9) node{\tiny{$\bullet$}};
\draw (3.5,5.2) node{\tiny{$\bullet$}};
\draw (3.8,5.2) node{\tiny{$\bullet$}};
\draw (4.1,5.2) node{\tiny{$\bullet$}};
\draw (4.4,5.2) node{\tiny{$\bullet$}};
\draw (4.7,5.2) node{\tiny{$\bullet$}};
\draw (7.7,5.2) node{\tiny{$\bullet$}};
\draw (8.0,5.2) node{\tiny{$\bullet$}};
\draw (8.3,5.2) node{\tiny{$\bullet$}};
\draw (3.5,5.5) node{\tiny{$\bullet$}};
\draw (3.8,5.5) node{\tiny{$\bullet$}};
\draw (4.1,5.5) node{\tiny{$\bullet$}};
\draw (4.4,5.5) node{\tiny{$\bullet$}};
\draw (4.7,5.5) node{\tiny{$\bullet$}};
\draw (7.7,5.5) node{\tiny{$\bullet$}};
\draw (8.0,5.5) node{\tiny{$\bullet$}};
\draw (8.3,5.5) node{\tiny{$\bullet$}};
\draw (3.5,5.8) node{\tiny{$\bullet$}};
\draw (3.8,5.8) node{\tiny{$\bullet$}};
\draw (4.1,5.8) node{\tiny{$\bullet$}};
\draw (4.4,5.8) node{\tiny{$\bullet$}};
\draw (4.7,5.8) node{\tiny{$\bullet$}};
\draw (7.7,5.8) node{\tiny{$\bullet$}};
\draw (8.0,5.8) node{\tiny{$\bullet$}};
\draw (8.3,5.8) node{\tiny{$\bullet$}};
\fill[color=gray!20] (5,3.53) rectangle (7,4.5);
\draw (5,3.53) rectangle (7,4.5);
\draw (5.0,3.7) node{\tiny{$\bullet$}};
\draw (5.3,3.7) node{\tiny{$\bullet$}};
\draw (5.6,3.7) node{\tiny{$\bullet$}};
\draw (5.9,3.7) node{\tiny{$\bullet$}};
\draw (6.2,3.7) node{\tiny{$\bullet$}};
\draw (6.5,3.7) node{\tiny{$\bullet$}};
\draw (6.8,3.7) node{\tiny{$\bullet$}};
\draw (7.1,3.7) node{\tiny{$\bullet$}};
\draw (7.4,3.7) node{\tiny{$\bullet$}};
\draw (5.0,4) node{\tiny{$\bullet$}};
\draw (5.3,4) node{\tiny{$\bullet$}};
\draw (5.6,4) node{\tiny{$\bullet$}};
\draw (6.2,4) node{\tiny{$\bullet$}};
\draw (6.5,4) node{\tiny{$\bullet$}};
\draw (6.8,4) node{\tiny{$\bullet$}};
\draw (7.1,4) node{\tiny{$\bullet$}};
\draw (7.4,4) node{\tiny{$\bullet$}};
\draw (5.0,4.3) node{\tiny{$\bullet$}};
\draw (5.3,4.3) node{\tiny{$\bullet$}};
\draw (5.6,4.3) node{\tiny{$\bullet$}};
\draw (5.9,4.3) node{\tiny{$\bullet$}};
\draw (6.2,4.3) node{\tiny{$\bullet$}};
\draw (6.5,4.3) node{\tiny{$\bullet$}};
\draw (6.8,4.3) node{\tiny{$\bullet$}};
\draw (7.1,4.3) node{\tiny{$\bullet$}};
\draw (7.4,4.3) node{\tiny{$\bullet$}};
\draw (5.0,4.6) node{\tiny{$\bullet$}};
\draw (5.3,4.6) node{\tiny{$\bullet$}};
\draw (5.6,4.6) node{\tiny{$\bullet$}};
\draw (5.9,4.6) node{\tiny{$\bullet$}};
\draw (6.2,4.6) node{\tiny{$\bullet$}};
\draw (6.5,4.6) node{\tiny{$\bullet$}};
\draw (6.8,4.6) node{\tiny{$\bullet$}};
\draw (7.1,4.6) node{\tiny{$\bullet$}};
\draw (7.4,4.6) node{\tiny{$\bullet$}};
\draw (5.0,4.9) node{\tiny{$\bullet$}};
\draw (5.3,4.9) node{\tiny{$\bullet$}};
\draw (5.6,4.9) node{\tiny{$\bullet$}};
\draw (5.9,4.9) node{\tiny{$\bullet$}};
\draw (6.2,4.9) node{\tiny{$\bullet$}};
\draw (6.5,4.9) node{\tiny{$\bullet$}};
\draw (6.8,4.9) node{\tiny{$\bullet$}};
\draw (7.1,4.9) node{\tiny{$\bullet$}};
\draw (7.4,4.9) node{\tiny{$\bullet$}};
\draw (5.0,5.2) node{\tiny{$\bullet$}};
\draw (5.3,5.2) node{\tiny{$\bullet$}};
\draw (5.6,5.2) node{\tiny{$\bullet$}};
\draw (5.9,5.2) node{\tiny{$\bullet$}};
\draw (6.2,5.2) node{\tiny{$\bullet$}};
\draw (6.5,5.2) node{\tiny{$\bullet$}};
\draw (6.8,5.2) node{\tiny{$\bullet$}};
\draw (7.1,5.2) node{\tiny{$\bullet$}};
\draw (7.4,5.2) node{\tiny{$\bullet$}};
\draw (5.0,5.5) node{\tiny{$\bullet$}};
\draw (5.3,5.5) node{\tiny{$\bullet$}};
\draw (5.6,5.5) node{\tiny{$\bullet$}};
\draw (5.9,5.5) node{\tiny{$\bullet$}};
\draw (6.2,5.5) node{\tiny{$\bullet$}};
\draw (6.5,5.5) node{\tiny{$\bullet$}};
\draw (6.8,5.5) node{\tiny{$\bullet$}};
\draw (7.1,5.5) node{\tiny{$\bullet$}};
\draw (7.4,5.5) node{\tiny{$\bullet$}};
\draw (5.0,5.8) node{\tiny{$\bullet$}};
\draw (5.3,5.8) node{\tiny{$\bullet$}};
\draw (5.6,5.8) node{\tiny{$\bullet$}};
\draw (5.9,5.8) node{\tiny{$\bullet$}};
\draw (6.2,5.8) node{\tiny{$\bullet$}};
\draw (6.5,5.8) node{\tiny{$\bullet$}};
\draw (6.8,5.8) node{\tiny{$\bullet$}};
\draw (7.1,5.8) node{\tiny{$\bullet$}};
\draw (7.4,5.8) node{\tiny{$\bullet$}};

\draw (6,4.05) node{$Q^4$};
\draw (6,9) node{\tiny{$\bullet$}};
\draw (6,9.3) node{$(0,0)$};
\draw[dashed] (6,0) -- (6,9);
\end{tikzpicture}
\end{center}
\begin{rmrk}\label{stackrmrk}
We will mention some properties of stack of cubes.
\begin{itemize}
   \item $Q^{k}(\qr(x_0,t_0))$ may not be unique.
   \item Radius of $Q^{k}(\qr(x_0,t_0))$ is $z_k z_{k-1}...z_{1}\rho$. 
    \item $Q^{k+1}(\qr(x_0,t_0)) \subset \tilde{Q}^{k}(\qr(x_0,t_0))$. We would like to remind the reader that $\tilde{Q}^{k}(\qr(x_0,t_0))$ corresponding to $Q^{k}(\qr(x_0,t_0))$ as defined in (\ref{tildqk})  is the "shifted in time predecessor"  and $Q^{k+1}(\qr(x_0,t_0))$ is the next member in the family of cubes.
    \item If $d$ is the distance between $\{(0,t):t \in \R \}$ and $Q^{k}(\qr(x_0,t_0))$, then the distance between $\{(0,t):t \in \R \}$ and $\tilde{Q}^{k}(\qr(x_0,t_0))$ is at most  max$\{0,d-2r_k\}$. Hence, the distance between $\{(0,t):t \in \R \}$ and $Q^{k+1}(\qr(x_0,t_0))$ is at most max$\{0,d-2r_k\}$.
\end{itemize}
\end{rmrk}
Now, we will specify $z_1$, $z_2$, . . ., $z_m$ for defining our stack of cubes corresponding to the given nonlinearity. To do this, we will first
 define, for $k \geq 1$,
 \begin{align*}
     a_k=\frac{1}{kL_2(\eta(L^k)+1)},
 \end{align*}
 where $L$ is from Lemma \ref{measureuniform}.
Then, we have
\begin{align*}
     \frac{a_k}{a_{k+1}}=\frac{(k+1)(\eta(L^{k+1})+1)}{k(\eta(L^k)+1)}\leq\frac{(k+1)\eta(LL^{k})}{k\eta(L^k)}+\frac{k+1}{k(\eta(L^k)+1)}\leq\frac{(k+1)\eta(LL^{k})}{k\eta(L^k)}+\frac{k+1}{2k},
 \end{align*}
 where last inequality is a consequence of $\eta(t)\geq1$ for all $t\geq0$.\\
 Using the fact that $(k+1)/k$ goes to $1$ as $k\rightarrow\infty$ and (i) of Lemma \ref{etaprop} with $c=L$, we can choose $k_1 \geq 4$ such that
 \begin{align}\label{k1}
     \underset{k\geq k_1}{\text{sup}}\frac{a_k}{a_{k+1}}\leq\underset{k\geq k_1}{\text{sup}}\frac{(k+1)\eta(LL^{k})}{k\eta(L^k)}+\frac{k+1}{2k}\leq2.
 \end{align}
 We also define
 \begin{align*}
     m_{k_1}:=3,\hspace{2mm}m_{k}:=\frac{a_{k-1}}{a_k} \hspace{2mm} \text{for} \hspace{2mm} k>k_1.
 \end{align*}
 Now, we are ready to define $z_i$'s.\\  If $Q_\rho(x_0,t_0)$ is such that $|Q_\rho(x_0,t_0) \cap \{{u} > L^{l}\}| > (1-\mu)|Q_\rho(x_0,t_0)|,$ then
 \begin{align*}
     z_i=m_{l-i} \hspace{2mm} \text{for} \hspace{2mm} 1 \leq i \leq l-{k_1}.
 \end{align*}
Note that $z_i$'s are depending on super level set at which measure density of cube has crossed. To emphasize this fact,
we will denote $Q^{k}(\qr(x_0,t_0))$ by $Q^{k}(\qr(x_0,t_0),l)$.

To prove Theorem \ref{mthm} we require a couple of lemmas. In the lemmas, we will assume $u$ is a solution of following equations:
\begin{equation}\label{meq1}
    P^-(D^2u)-u_t\leq \tilde{\phi}(|Du|)
\end{equation}
and
\begin{equation}\label{meq2}
    P^+(D^2u)-u_t\geq-\tilde{\phi}(|Du|)
\end{equation}
where $\tilde{\phi} = {\Lambda_{0}}{r_1} \eta \big(\frac{1}{r_1}\big) \phi$ and $r_1$ is from Lemma \ref{measureuniform}.

Now, we will prove the lemma which states how measure information will propagate.  
\begin{lemma}\label{prop}
Let $u\in C(Q_2)$ be a viscosity supersolution of (\ref{meq1}) in $Q_2$.
 Let $ Q_\rho(x_0,t_0)$ satisfy the following conditions:
\begin{itemize}
    \item$|Q_\rho(x_0,t_0) \cap \{\tilde{u} > L^{k+1}\}| > \big( 1-\mu \big) |Q_\rho(x_0,t_0)|$,
    \item $ \rho \leq \frac{c_n{a_k}}{r}$,
\end{itemize}
where $\tilde{u}(x,t)=u(rx,r^2t).$
Then, the following hold:
\begin{itemize}
    \item  We have
    \begin{align*}
        \tilde{u}(x,t)> L^{k} \hspace{2mm}
        \text{in}  \hspace{2mm} \{(x,t): |x-x_0|_{\infty}<3\rho,t_0<t<(c_n^{-1}\rho)^2+t_0-\rho^2\}\cap\{-1 < t < 0\},
    \end{align*}
     which in particular implies,
     \begin{align*}
        \tilde{u}(x,t)> L^{k} \hspace{2mm}
        \text{in}  \hspace{2mm}  \widetilde{Q}^0(Q_{\rho}(x_0,t_0),k+1)\cap\{-1 < t < 0\}.
     \end{align*}
    
    \item For $k > k_1$, we have $\tilde{u}> L^{k-i}$ in $\widetilde{Q}^i(Q_{\rho}(x_0,t_0),k+1)\cap\{-1 < t < 0\}$ for all $i$ such that $0 \leq i < k-k_1$.
\end{itemize}
where ${Q^i}(Q_{\rho}(x_0,t_0),k+1)$ denotes the $i^{th}$ member of stacks of cube corresponding to $Q_{\rho}(x_0,t_0)$.
\end{lemma}
\begin{rmrk}
We can rephrase the first assertion of above lemma as follows: Take $(\hat{x},\hat{t} )$ such that $Q_{\rho}(x_0,t_0) =(\hat{x},\hat{t})+(c_n^{-1}\rho)K_1$, then 
\begin{align*}
    \tilde{u} > L^k \hspace{2mm} \text{in} \hspace{2mm} ((\hat{x},\hat{t})+(c_n^{-1}\rho)K_3) \cap \{-1 <t <0\}.
\end{align*}

\end{rmrk}
\begin{proof}
We will prove the result by induction on $i$.

For the case $i=0$, if
\begin{align*}
   \{(x,t):|x-x_0|_{\infty}<3\rho, t_0 <t<(c_n^{-1}\rho)^2+t_0-\rho^2\}\cap\{-1 < t < 0\}=\emptyset,
\end{align*}
then there is nothing to prove. So, we will assume
\begin{align}\label{200}
    \{(x,t):|x-x_0|_{\infty}<3\rho, t_0<t<(c_n^{-1}\rho)^2+t_0-\rho^2\}\cap\{-1< t < 0\} \not= \emptyset.
\end{align}
Define 
\begin{align*}
   v(\,x, t )\, &=\frac{\tilde{u}(c_n^{-1}\rho  x + x_0, {(\,c_n^{-1} \rho)\,}^2 t +t_0-\rho^2+{(\,c_n^{-1} \rho)\,}^2)\,}{L^k} ;\\
   &=\frac{u(c_n^{-1}\rho r x + x_0, {(\,c_n^{-1} \rho)\,}^2 r^2t +t_0-\rho^2+{(\,c_n^{-1} \rho)\,}^2)\,}{L^k} ;
\end{align*}
where $(x,t) \in D:=\{(x,t):|x|_{\infty}<1,-1<t\leq -1+c_n^2-t_0{(\,c_n^{-1} \rho)\,}^2\}$.\\
From (\ref{200}) we have $t_0 <0$, which implies $D$ contains $K_1$.
It is given that $\rho \leq \frac{c_na_k}{r}$, i.e., $c_n^{-1}\rho r \leq a_k$. So, by Lemma \ref{scaling}, $v$ is a supersolution in $D$.\\ Clearly, $(x,t) \in K_1$ iff $(c_n^{-1}\rho x + x_0, {(\,c_n^{-1}\rho)\,}^2t +t_0-\rho^2+{(\,c_n^{-1} \rho)\,}^2) \in Q_\rho(x_0,t_0)$, so
\begin{align*}
    \frac{|K_1 \cap \{v > L\}|}{|K_1|}=\frac{|Q_\rho(x_0,t_0) \cap \{\tilde{u} > L^{k+1}\}|}{|Q_\rho(x_0,t_0)|}.
\end{align*}
Now use first hypothesis of lemma to get $|K_1 \cap \{v > L\}| > \big( 1-\mu \big) |K_1|$, which using Remark \ref{rmu} gives ${\inf }_{K_3 \cap D} v > 1$. 
By rescaling back, we get the first part of the lemma. In particular, it gives the conclusion for $i=0.$

Assume result holds for $i\leq l$. We will prove it for $i=l+1 < k-k_1$.\\
If $\widetilde{Q}^{l+1}(Q_{\rho}(x_0,t_0),k+1)\cap\{-1 < t < 0\}= \emptyset$ then there is nothing to prove and we are done.\\
So, we will proceed for the case where $\widetilde{Q}^{l+1}(Q_{\rho}(x_0,t_0),k+1)\cap\{-1< t < 0\}\not= \emptyset$. In this case, 
${Q}^{l+1}(Q_{\rho}(x_0,t_0),k+1)\subset Q_{2}$.
By induction hypothesis, we have
\begin{equation}\label{1}
    \tilde{u} > L^{k-l} \hspace{2mm}\text{in} \hspace{2mm} {Q}^{l+1}(Q_{\rho}(x_0,t_0),k+1).
\end{equation}
By definition, ${Q}^{l+1}(Q_{\rho}(x_0,t_0),k+1)$ is a cube of radius $z_{l+1}z_{l}...z_1\rho$, which is the same as $m_{k-l}m_{k+1-l}...m_{k}\rho$ (we will call it $r_{l+1}$) and we call its center $(\,x_1,t_1)\,$ i.e., ${Q}^{l+1}(Q_{\rho}(x_0,t_0),k+1)=Q_{r_{l+1}}(\,x_1,t_1)\,$.\\
Define  
\begin{align*}
    v(\,x, t )\,&=\frac{\tilde{u}(c_n^{-1}r_{l+1} x + x_1, {(\,c_n^{-1}r_{l+1})\,}^2t +t_1-{r_{l+1}}^2+{(\,c_n^{-1}r_{l+1})\,}^2)\,}{L^{k-l-1}}\\
    &= \frac{u(c_n^{-1}r_{l+1}r x + x_1, {(\,c_n^{-1}r_{l+1}r)\,}^2t +t_1-{r_{l+1}}^2+{(\,c_n^{-1}r_{l+1})\,}^2)\,}{L^{k-l-1}},
    \end{align*}
where $(x,t)\in D:=\{(x,t):|x|_{\infty}<1,-1<t\leq -1+c_n^2-t_1{(\,c_n^{-1}r_{l+1})\,}^2\}$.\\
Note that $\widetilde{Q}^{l+1}(Q_{\rho}(x_0,t_0),k+1)\cap\{-1 < t < 0\}\not= \emptyset$  implies $t_1 < 0$. Hence, $D$ contains $K_1$.
In view of Lemma \ref{scaling}, we know $v$ will be a supersolution in its domain if
\begin{align*}
    c_n^{-1}r_{l+1}r &\leq a_{k-l-1},\\
   \text{i.e.,} \hspace{2mm} c_n^{-1} z_{l+1}z_{l}...z_1\rho r &\leq a_{k-l-1},\\
 \text{i.e.,} \hspace{2mm} c_n^{-1}m_{k-l}m_{k-l+1}...m_{k}\rho r &\leq a_{k-l-1},\\
  \text{i.e.,} \hspace{2mm} c_n^{-1}\frac{a_{k-l-1}}{a_{k-l}}\frac{a_{k-l}}{a_{k-l+1}}...\frac{a_{k-1}}{a_{k}}\rho r &\leq a_{k-l-1},\\
  \text{i.e.,} \hspace{2mm} c_n^{-1} \rho r &\leq a_{k}.
\end{align*}
This is true by the second hypothesis of lemma. Hence, $v$ is a supersolution in its domain. Also, we have $(x,t) \in K_1$ iff $(c_n^{-1}r_{l+1} r x + x_1, {(\,c_n^{-1}r_{l+1} r)\,}^2t +t_1-{r_{l+1}}^2+{(\,c_n^{-1}r_{l+1})\,}^2))\, \in Q_{r_{l+1}}(\,x_1,t_1)\,$ which gives
\begin{align*}
    \frac{|K_1 \cap \{v > L\}|}{|K_1|}=\frac{|Q_{r_{l+1}}(\,x_1,t_1)\, \cap \{\tilde{u} > L^{k-l}\}|}{|Q_{r_{l+1}}(\,x_1,t_1)\,|}.
\end{align*}
The above equation using \eqref{1} gives $|K_1 \cap \{v > L\}| = |K_1| > \big( 1-\mu \big) |K_1|$. Now, use Remark \ref{rmu} to get $\underset{K_3 \cap D}{\inf v} > 1$, which in particular implies $\tilde{u}>L^{k-l-1}$ in $\widetilde{Q}^{l+1}(Q_{\rho}(x_0,t_0),k+1)\cap\{-1 < t < 0\}$. This proves the induction step and hence the lemma.
\end{proof}
In the above lemma, measure information was propagating provided initial cube's size is small. We will now prove the lemma which will ensure smallness of cube.

\begin{lemma}\label{decay}
 Let $u\in C(Q_2)$ be a viscosity supersolution of (\ref{meq1}) in $Q_2$ such that
 $u(\,0,0)\, = 1$. 
Let there exist $Q_\rho(x_1,t_1) \subset K_1$ such that
\begin{align}\label{aa}
  | Q_\rho(x_1,t_1) \cap \{\tilde{u} > L^k\}| > \big( 1-\mu \big) |Q_\rho(x_1,t_1)|,
\end{align}
where $\tilde{u}=u(a_{k_1}x,a_{k_1}^2t)$.
Then, 
$\hspace{2mm} \rho < \frac{ c_n a_{k-1}}{3 a_{k_1}}$ for all $k \geq k_1 +1$.
\end{lemma}
\begin{proof}
We will prove it by induction on $k$.\\
Step 1: For $k=k_1+1$.
Since $Q_\rho(x_1,t_1) \subset K_1$ so $\rho \leq c_n$, which is the same as $\rho \leq \frac{c_n a_{k_1}}{a_{k_1}}$.
Then by Lemma \ref{prop}, we have 
\begin{align}\label{ak2}
    \tilde{u}(x,t) &>L^{k_1} \hspace{2mm} \text{in} \hspace{2mm}\tilde{Q}_{\rho}(x_1,t_1).
 \end{align}
We claim that $3\rho < c_n$.\\
Assume this is not true, i.e., $3\rho \geq c_n$. 
Recall radius of $\tilde{Q}_{\rho}(x_1,t_1)$ is $3\rho$, which is greater than or equal to $c_n$, so it contains some translated version of $K_1$. So, we can
take $(\,x_0,t_0)\,$ such that the following hold:
\begin{itemize}
    \item[(i)] $(\,x_0,t_0)\, +K_1 \subset \tilde{Q}_{\rho}(x_1,t_1)$,
    \item[(ii)]$\{(0,s):s<0\} \cap (\,x_0,t_0)\, +K_1 \not= \emptyset.$
\end{itemize}
(we can assure the condition (ii) by a simple application of triangle inequality to the fact that $Q_\rho(x_1,t_1) \subset K_1
$) \\
Also, note that condition (ii) will imply $(0,0) \in (\,x_0,t_0)\, + K_3$. Note that (\ref{ak2}) gives
\begin{align*}
   \frac{|(x_0,t_0)+K_1 \cap \{\tilde{u} > L^{k_1}\}|}{|(x_0,t_0)+K_1|}=1>(1-\mu).
\end{align*}
Since $(a_k)$ is a decreasing sequence, $c_n < \frac{c_n a_{k_1-1}}{a_{k_1}}$. Hence, by Lemma \ref{prop}, we get $\tilde{u} > L^{k_1-1}$ in $((x_0,t_0)+K_3) \cap Q_1$. In particular, we get $u(0,0) > 1$, which is a contradiction to $u(0,0)=1$.
Hence we get our claim, i.e., $3\rho < c_n$.
This proves the lemma for $k=k_1+1.$\\
Step 2: Assume result holds for $k=l \geq k_1+1$, i.e., if there exists $Q_\rho \subset K_1$ such that
\begin{align}\label{aa1}
  | Q_\rho(x_1,t_1) \cap \{\tilde{u} > L^l\}| > \big( 1-\mu \big) |Q_\rho(x_1,t_1)|
\end{align}
then $\hspace{2mm} \rho < \frac{ c_n a_{l-1}}{3 a_{k_1}}$.\\
Take $Q_\rho(x_1,t_1) \subset K_1$ such that 
\begin{align}\label{aa2}
  | Q_\rho(x_1,t_1) \cap \{\tilde{u} > L^{l+1}\}| > \big( 1-\mu \big) |Q_\rho(x_1,t_1)|.
\end{align}
Clearly, (\ref{aa2}) implies (\ref{aa1}). So, by induction hypothesis, we have 
\begin{align*}
    \rho < \frac{ c_n a_{l-1}}{3 a_{k_1}}.
\end{align*}
Using $\underset{k\geq k_1}{\text{sup}}\frac{a_k}{a_{k+1}} \leq 2$ and $l-1 \geq k_1$, We get
\begin{align}\label{aa3}
    \rho < \frac{ c_n a_{l}}{ a_{k_1}}.
\end{align}
Now (\ref{aa2}) and (\ref{aa3}) make us eligible to apply Lemma \ref{prop}, and we get
\begin{align}\label{aa500}
    \tilde{u}> L^{l-i} \hspace{2mm}\text{in}\hspace{2mm} \widetilde{Q}^i(Q_{\rho}(x_1,t_1),l+1)\cap\{-1 < t < 0\} \hspace{2mm}\text{for all}\hspace{2mm} i\hspace{2mm} \text{such that}\hspace{2mm} 0 \leq i < l-k_1,
\end{align}
which implies,
\begin{align}\label{aa501}
     \tilde{u}> L^{l+1-i} \hspace{2mm}\text{in}\hspace{2mm} {Q}^i(Q_{\rho}(x_1,t_1),l+1)\cap\{-1 < t < 0\} \hspace{2mm}\text{for all}\hspace{2mm} i\hspace{2mm} \text{such that}\hspace{2mm} 0 \leq i \le l-k_1.
\end{align}
If $\widetilde{Q}^{l-k_1}(Q_{\rho}(x_1,t_1),l+1)\cap\{-1 < t < 0\} \not= \emptyset$, then ${Q}^{l-k_1}(Q_{\rho}(x_1,t_1),l+1) \subset Q_2.$\\
Also, note that the radius of ${Q}^{l-k_1}(Q_{\rho}(x_1,t_1),l+1)$ is $z_{l-k_1}z_{l-k_1-1}...z_1\rho$,  i.e., $m_{k_1+1}m_{k_1+2}...m_l\rho$. Using (\ref{aa3}) we get
\begin{align*}
   \text{radius of ${Q}^{l-k_1}(Q_{\rho}(x_1,t_1),l+1)$}=\frac{a_{k_1}}{a_l}\rho<c_n
      \leq \frac{c_n a_{k_1}}{a_{k_1}}.
\end{align*}
Now, apply Lemma \ref{prop} ( for $k=k_1$ and cube ${Q}^{l-k_1}(Q_{\rho}(x_1,t_1),l+1)$ ) to get 
\begin{align}\label{aa502}
    \tilde{u}>L^{k_1} \hspace{2mm} \text{in} \hspace{2mm}\widetilde{Q}^{l-k_1}(Q_{\rho}(x_1,t_1),l+1)\cap\{-1 < t < 0\}. 
\end{align}
Hence, from (\ref{aa500}), (\ref{aa501}) and (\ref{aa502}), we have
\begin{align}\label{aa4}
    \tilde{u}> L^{l-i} \hspace{2mm}\text{in}\hspace{2mm} \widetilde{Q}^i(Q_{\rho}(x_1,t_1),l+1)\cap\{-1 < t < 0\} \hspace{2mm}\text{for all}\hspace{2mm} i\hspace{2mm} \text{such that}\hspace{2mm} 0 \leq i \leq l-k_1,
\end{align}
which implies,
\begin{align}\label{aa5}
    \tilde{u}> L^{l+1-i} \hspace{2mm}\text{in}\hspace{2mm} {Q}^i(Q_{\rho}(x_1,t_1),l+1)\cap\{-1 < t < 0\} \hspace{2mm}\text{for all}\hspace{2mm} i\hspace{2mm} \text{such that}\hspace{2mm} 0 \leq i \le l-k_1+1.
\end{align}
\textbf{Claim(1)}: $\rho < \frac{ c_n a_{l}}{3 a_{k_1}}$.\\
\textbf{Proof of Claim(1)}: We will prove it by contradiction. Assume Claim(1) is not true, i.e., $\rho \geq \frac{ c_n a_{l}}{3 a_{k_1}}$,  which is the same as saying that,  $\frac{3 a_{k_1}}{  a_{l}}\rho \geq c_n$, which we can write as $3\frac{a_{k_1}}{a_{k_1+1}}\frac{a_{k_1+1}}{a_{k_1+2}}...\frac{a_{l-1}}{a_{l}}\rho \geq c_n$, which is $m_{k_1}m_{k_1+1}...m_{l}\rho \geq c_n$, i.e., radius of ${Q}^{l+1-k_1}(Q_{\rho}(x_1,t_1),l+1) \geq c_n.$ Choose first $i_0 $ such that radius of ${Q}^{i_0}(Q_{\rho}(x_1,t_1),l+1) \geq c_n.$
Clearly, $i_0 \leq l+1-k_1$. \\
There are three possible situations:
\begin{itemize}
    \item[1.] ${Q}^{i_0}(Q_{\rho}(x_1,t_1),l+1) \subset \R^n\times \{t<0\}$.
     \item[2.] ${Q}^{i_0}(Q_{\rho}(x_1,t_1),l+1) \cap \R^n\times \{t=0\} \not= \emptyset$.
      \item[3.] ${Q}^{i_0}(Q_{\rho}(x_1,t_1),l+1) \subset \R^n\times \{t>0\}$.
\end{itemize}
Now, we will rule out all three situations to get a contradiction.

\textbf{Case 1}: ${Q}^{i_0}(Q_{\rho}(x_1,t_1),l+1) \subset \R^n\times \{t<0\}$.

Note that radius of ${Q}^{i_0}(Q_{\rho}(x_1,t_1),l+1) \geq c_n$, so we can take $(x_0,t_0)$ such that
\begin{itemize}
    \item[(A)] $(x_0,t_0)+K_1 \subset {Q}^{i_0}(Q_{\rho}(x_1,t_1),l+1)$ and
    \item[(B)] $(x_0,t_0)+K_1$ intersects $\{0\} \times \R$.
\end{itemize}
Note that radius of ${Q}^{i_0}(Q_{\rho}(x_1,t_1),l+1) \geq c_n$ and $Q_{\rho}(x_1,t_1) \subset K_1$. With an easy application of triangle inequality to these facts, we can ensure condition (B). Condition (B) will further imply $(0,0) \in (x_0,t_0)+K_3$.

Note that, from (\ref{aa5}) and $i_0 \leq l+1-k_1$, we have $\tilde{u}>L^{k_1}$ in $(x_0,t_0)+K_1$. Also, radius of $(x_0,t_0)+K_1=c_n \leq \frac{c_n a_{k_1-1}}{a_{k_1}}.$ Therefore, from Lemma \ref{prop}, we get $\tilde{u}>L^{k_1-1}$ in $(x_0,t_0)+K_3 \cap Q_1$, which in particular implies $u(0,0)=\tilde{u}(0,0)>L$, but this is not true as $u(0,0)=1$. Hence, this situation is ruled out.

\textbf{Case 2}: ${Q}^{i_0}(Q_{\rho}(x_1,t_1),l+1) \cap \R^n\times \{t=0\} \not= \emptyset$.

By definition of $i_0$, we have  ${Q}^{i_0}(Q_{\rho}(x_1,t_1),l+1) \cap \{0\}\times \R \not= \emptyset$. Hence, we get $(0,0) \in {Q}^{i_0}(Q_{\rho}(x_1,t_1),l+1).$ Use (\ref{aa5}) and $i_0 \leq l+1-k_1$ to get $u(0,0)=\tilde{u}(0,0)>L^{k_1}$, which is false as $u(0,0)=1$. Hence, this situation is ruled out.\\
\textbf{Case 3}: ${Q}^{i_0}(Q_{\rho}(x_1,t_1),l+1) \subset \R^n\times \{t>0\}$.

\textbf{Claim(2)}: There exists $i<i_0$ such that $(0,0) \in {Q}^{i}(Q_{\rho}(x_1,t_1),l+1)$.

Assume Claim(2) has been established, then by (\ref{aa5}) we have  $u(0,0) > L^{l+1-i}$. Now, $i \leq l-k_1$ implies $u(0,0) > L^{k_1+1}$, which contradicts $u(0,0)=1$. Hence, we will rule this situation out provided we could establish the Claim(2). 

Now we will establish Claim(2).\\
Choose first $j_0$ such that ${Q}^{j_0}(Q_{\rho}(x_1,t_1),l+1)\cap \{0\}\times \R \not= \emptyset$. Clearly, $j_0 \leq i_0$.
Note the following observations:
\begin{itemize}
    \item 
    For $j_0 \geq 2$, we have
    \begin{align*}
       r_0+r_1+...+r_{j_0-2}<\frac{c_n\sqrt{n}}{2}.
   \end{align*}
The reason is as follows. If $d$ is the distance between $Q_{\rho}(x_1,t_1)$ and $\{0\}\times \R$, then by Remark \ref{stackrmrk},   distance between ${Q}^{j_0-1}(Q_{\rho}(x_1,t_1),l+1)$ and $\{0\}\times \R$ is at most max$\{d-2r_0-2r_1-...-2r_{j_0-2},0\}$. Now,
    \begin{align*}
        {Q}^{j_0-1}(Q_{\rho}(x_1,t_1),l+1)\cap \{0\}\times \R = \emptyset\\
        \implies d-2r_0-2r_1-...-2r_{j_0-2} > 0\\
        \implies 2r_0+2r_1+...+2r_{j_0-2}<d.
    \end{align*}
   Also, $Q_{\rho}(x_1,t_1) \subset K_1$ implies $d<c_n\sqrt{n}$. Hence, we get 
   \begin{align*}
       r_0+r_1+...+r_{j_0-2}<\frac{c_n\sqrt{n}}{2}.
   \end{align*}
   \item For $j_0 \geq 1$, $r_{j_0-1} < c_n$ and $r_{j_0}< 3c_n$.\\
   ($r_{j_0-1} < c_n$ follows from definition of $j_0$. And $r_{j_0}< 3c_n$ follows from $r_{j_0} \leq 3r_{j_0-1}$ )
\end{itemize}
Now, we will show ${Q}^{j_0}(Q_{\rho}(x_1,t_1),l+1) \subset \R^n \times \{t<0\}$. For $j_0=0$, we are done because $Q_{\rho}(x_1,t_1) \subset K_1$. Also, for $j_0 \geq 1$, we have $(x,t) \in {Q}^{j_0}(Q_{\rho}(x_1,t_1),l+1)$ implies
\begin{align*}
    t  
    \leq &\hspace{1mm}t_1+r_1^2+...+r_{j_0-2}^2+r_{j_0-1}^2+r_{j_0}^2\\
    \leq&-1+ c_n^2+r_1^2+...+r_{j_0-2}^2+r_{j_0-1}^2+r_{j_0}^2\\
    <&-1+ c_n^2+r_0^2+r_1^2+...+r_{j_0-2}^2+r_{j_0-1}^2+r_{j_0}^2\\
    \leq&-1+ c_n^2+(r_0+r_1+...+r_{j_0-2})^2+r_{j_0-1}^2+r_{j_0}^2\\
    \leq&-1+ c_n^2+\frac{nc_n^2}{4}+c_n^2+9c_n^2\\
    \leq&-1+12nc_n^2\\
    \leq&-1+\frac{12}{100n}\\
    <&0.
\end{align*}
Therefore we are done. Thus, we get $j_0<i_0$,
\begin{align*}
    {Q}^{j_0}(Q_{\rho}(x_1,t_1),l+1) \subset \R^n\times \{t<0\},\\
    {Q}^{j_0}(Q_{\rho}(x_1,t_1),l+1)\cap \{0\}\times \R \not= \emptyset.
\end{align*}
Also, we have ${Q}^{i_0}(Q_{\rho}(x_1,t_1),l+1) \subset \R^n\times \{t>0\}$. Therefore, the way we have defined the stack of cubes will give $i$ such that $j_0<i<i_0$ and $(0,0) \in {Q}^{i}(Q_{\rho}(x_1,t_1),l+1)$. This establishes the Claim(2).
\end{proof}
 
Now, we are ready to prove the $L^{\epsilon}$-estimate.

\begin{lemma}\label{wh}
 Let $u\in C(Q_2)$ be a viscosity supersolution of (\ref{meq1}) in $Q_2$ such that
 $u(\,0,0)\, = 1$. Then, there exist universal constants $C$ and $\epsilon >0$ such that the following hold:
 \begin{align}
    |\{\tilde{u} >\tau\} \cap \hat{K}| \leq C\tau^{-{\epsilon}}
\end{align}
and
 \begin{align}\label{lepsilon}
     \bigg(\int_{\hat{K}}\tilde{u}^{\epsilon}\bigg)^{\frac{1}{\epsilon}} \leq C,
 \end{align}
 where $\tilde{u}(x,t)=u(a_{k_1}x,a_{k_1}^2t),$ and
 $\hat{K} = \{(x,t):|x|_{\infty} < c_n, -1 < t < -1+c_n^2/2\}$.
\end{lemma}
\begin{proof}
 Choose $m$ large enough so that $\big(1-\frac{\mu}{2}\big) \leq \frac{m+1}{m}(1-\mu)$, where $\mu$ is from Lemma \ref{measureuniform}. Now, choose $k_{2} >k_1$ large enough such that 
\begin{equation}
 \sum_{k\geq k_2} \frac{(m+1) a_k^2}{a_{k_1}^2}= \frac{m+1}{a_{k_1}^2} \sum_{k\geq k_2}a_k^2 \leq \frac{m+1}{a_{k_1}^2} \sum_{k\geq k_2} \frac{1}{k^2} \leq \frac{c_n^2}{2}.
\end{equation}
Define, for $k$ $\geq$ $k_2$,
\begin{equation}
    A_k = \bigg\{(x,t):\hspace{1mm} \tilde{u}(x,t) > L^{km},\hspace{1mm} |x|_{\infty} < c_n \hspace{2mm} \& \hspace{2mm} -1 < t < -1+c_n^2-   \frac{m+1}{a_{k_1}^2}\sum_{j=k_2}^{k} \frac{1}{j^2}  \bigg\}.
\end{equation}
We claim that, for $k \geq k_2$, the following holds:
\begin{align*}
    | A_{k+1} | \leq \Big( 1- \frac{\mu}{2} \Big) | A_{k} |.
\end{align*}
To prove this, we want to apply the Lemma \ref{cald} with 
\begin{align*}
    A&=A_{k+1},\\
    B&=A_{k},\\ 
    Q&=K_1=(-c_n,c_n)^n \times (-1,-1+c_n^2),\\
 \text{and} \hspace{2mm}  \delta&=1-\mu.
\end{align*}
We have $\tilde{u}(0,0)=u(0,0)=1$, so $\text{inf}_{K_3}\tilde{u}\leq 1$. Hence, by Lemma \ref{measureuniform}, we have 
\begin{align*}
    |A| \leq |\{\tilde{u} > L\} \cap K_1| < (1-\mu)|K_1|=(1-\mu)|Q|.
\end{align*}
Take any dyadic cube $Q_\rho(x_1,t_1) \subset Q$ such that
\begin{equation*}
| Q_\rho(x_1,t_1) \cap A | > ( 1-\mu ) |Q_\rho(x_1,t_1)|,
\end{equation*}
i.e.,
\begin{equation}\label{11}
| Q_\rho(x_1,t_1) \cap A_{k+1} | > ( 1-\mu ) |Q_\rho(x_1,t_1)|.
\end{equation}
 From now onwards in this proof we will call $Q_\rho(x_1,t_1)$ by $Q_\rho$.
In view of Lemma \ref{cald}, we have to show $ \Bar{Q_{\rho}}^m \subset B$,
i.e., we have to show $\Bar{Q_{\rho}}^m \subset A_{k}$.\\
First we will show
\begin{align}\label{145}
   \Bar{Q_{\rho}}^m \subset \{|x|_{\infty} < c_n\} \times \bigg\{-1 < t < -1+c_n^2-   \frac{m+1}{a_{k_1}^2}\sum_{j=k_2}^{k} \frac{1}{j^2}  \bigg\}.
\end{align}
Because of (\ref{11}), we have
\begin{align*}
    Q_{\rho} \cap \{|x|_{\infty} < c_n\} \times \bigg\{-1 < t < -1+c_n^2-   \frac{m+1}{a_{k_1}^2}\sum_{j=k_2}^{k+1} \frac{1}{j^2}  \bigg\} \not= \emptyset.
\end{align*}
Hence
\begin{align*}
    \Bar{Q_{\rho}}^m \subset \{|x|_{\infty} < c_n\} \times \bigg\{-1 < t < -1+c_n^2-   \frac{m+1}{a_{k_1}^2}\sum_{j=k_2}^{k+1} \frac{1}{j^2} + \text{height}(\Bar{Q_{\rho}})+\text{height}(\Bar{Q_{\rho}}^m) \bigg\},
\end{align*}
where height($L$) = sup$\{t:\exists\hspace{1mm} x,\hspace{1mm}(x,t)\in L\}$ - inf$\{t:\exists\hspace{1mm} x,\hspace{1mm}(x,t)\in L\}$.
Moreover,
\begin{align*}
    \text{height}({Q_{\rho}})&=\rho^2,\\
    \text{height}(\Bar{Q_{\rho}})&=4\hspace{1mm} \text{height}({Q_{\rho}})=4\rho^2,\\
    \text{height}(\Bar{Q_{\rho}}^m)&=m\hspace{1mm}\text{height}(\Bar{Q_{\rho}})=4m\rho^2.
\end{align*}
To show (\ref{145}), it is enough to show
\begin{align*}
    -1+c_n^2-   \frac{m+1}{a_{k_1}^2}\sum_{j=k_2}^{k+1} \frac{1}{j^2} + \text{height}(\Bar{Q_{\rho}})+\text{height}(\Bar{Q_{\rho}}^m) \leq -1+c_n^2-   \frac{m+1}{a_{k_1}^2}\sum_{j=k_2}^{k} \frac{1}{j^2},
\end{align*}
which is the same as showing 
\begin{align*}
     \text{height}(\Bar{Q_{\rho}})+\text{height}(\Bar{Q_{\rho}}^m) \leq    \frac{m+1}{a_{k_1}^2} \frac{1}{(k+1)^2},
\end{align*}
which is the same as showing
\begin{align}\label{1452}
    4(m+1)\rho^2 \leq \frac{m+1}{a_{k_1}^2} \frac{1}{(k+1)^2}.
\end{align}
By Lemma \ref{decay}, (\ref{11}) implies
\begin{align*}
    \rho \leq \frac{c_n}{3a_{k_1}} a_{(k+1)m-1}.
\end{align*}
$\text{Using}\hspace{2mm} \underset{k\geq k_1}{\text{sup}}\frac{a_k}{a_{k+1}}\leq 2$, we get
\begin{align*}
    \rho \leq \frac{c_n}{a_{k_1}} a_{(k+1)m}.
\end{align*}
Since $a_k \leq \frac{1}{k}$, we have
\begin{align*}
    \rho \leq \frac{c_n}{a_{k_1}} \frac{1}{(k+1)m}.
\end{align*}
Using $4 c_n^2 =4 (10n)^{-2}\leq 1$, we get 
\begin{align*}
      4(m+1)\rho^2 \leq \frac{m+1}{a_{k_1}^2} \frac{1}{((k+1)m)^2}.
\end{align*}
Since $m \geq 1$, we get (\ref{1452}). Hence we get (\ref{145}).

Now we will show
\begin{align*}
     \tilde{u}>L^{km}\hspace{2mm}\text{in} \hspace{2mm}  \Bar{Q_{\rho}}^m.
\end{align*}
Since $Q_\rho(x_1,t_1) \subset K_1$ satisfies (\ref{11}), by Lemma \ref{decay}, we get $\rho < \frac{c_n}{3a_{k_1}} a_{(k+1)m-1}$.
Hence, we can apply Lemma \ref{prop} to get
\begin{align*}
    \tilde{u}>L^{(k+1)m-1} \hspace{2mm}\text{in} \hspace{2mm}  B_{3\rho}(x_1) \times (t_1,t_1+(c_n^{-1}\rho)^2-\rho^2) \cap \{-1<t<0\},
\end{align*}
where $B_{3\rho}(x_1):=\{x \in \R^n:|x-x_1|_{\infty}<3\rho\}$.

If \[\Bar{Q_{\rho}}^m \subset B_{3\rho}(x_1) \times (t_1,t_1+(c_n^{-1}\rho)^2-\rho^2) \cap \{-1<t<0\},\] then we are done.
If not, then note that (\ref{145}) gives 
\[D:=\{x:|x-x_1|_{\infty}<\rho\} \times (t_1+(c_n^{-1}\rho)^2-2\rho^2,t_1+(c_n^{-1}\rho)^2-\rho^2) \subset Q_1(0,0),\]
Also, we have 
\begin{align*}
    \rho &< \frac{c_n}{3a_{k_1}} a_{(k+1)m-1}\\
    & < \frac{c_n}{3a_{k_1}} a_{(k+1)m-2}\hspace{5mm} \text{(since $(a_k)$ is a decreasing sequence).}
\end{align*}
Now, apply Lemma \ref{prop} ( with $k=(k+1)m-2$ and cube $D$) to get
\begin{align*}
    \tilde{u}>L^{(k+1)m-2} \hspace{2mm}\text{in} \hspace{2mm} B_{3\rho}(x_1) \times (t_1+(c_n^{-1}\rho)^2-\rho^2,t_1+2[(c_n^{-1}\rho)^2-\rho^2]) \cap \{-1<t<0\}.
\end{align*}
Then, we have
\begin{align*}
    \tilde{u}>L^{(k+1)m-2} \hspace{2mm}\text{in} \hspace{2mm} B_{3\rho}(x_1) \times (t_1,t_1+2[(c_n^{-1}\rho)^2-\rho^2]).
\end{align*}

If $\Bar{Q_{\rho}}^m \subset B_{3\rho}(x_1) \times (t_1,t_1+2[(c_n^{-1}\rho)^2-\rho^2]) \cap \{-1<t<0\}$ then we are done.\\
If not, then continue like before and repeat the process.
Also, note that 
\begin{align*}
    m(c_n^{-2}-1)\rho^2 =m(100n^2-1) \geq 99m\rho^2 > 4(m+1)\rho^2.
\end{align*}
Therefore we get
\begin{align*}
    \Bar{Q_{\rho}}^m \subset B_{3\rho}(x_1) \times (t_1,t_1+j[(c_n^{-1}\rho)^2-\rho^2]) \cap \{-1<t<0\}
\end{align*}
for some $j \in \{1,2,3,...,m\}$ and
\begin{align*}
    \tilde{u}>L^{(k+1)m-j} \hspace{2mm}\text{in} \hspace{2mm} B_{3\rho}(x_1) \times (t_1,t_1+j[(c_n^{-1}\rho)^2-\rho^2]),
\end{align*}
which implies,
\begin{align*}
    \tilde{u}>L^{km} \hspace{2mm}\text{in} \hspace{2mm} B_{3\rho}(x_1)\times (t_1,t_1+j[(c_n^{-1}\rho)^2-\rho^2]).
\end{align*}
In particular, we have
\begin{align}\label{1453}
     \tilde{u}>L^{km}\hspace{2mm}\text{in} \hspace{2mm}  \Bar{Q_{\rho}}^m.
\end{align}
Thus, from (\ref{145}) and (\ref{1453}) we get $\Bar{Q_{\rho}}^m \subset A_k$.
Hence, by Lemma \ref{cald}, we get
\begin{align*}
    |A_{k+1}| &\leq (1-\mu)\frac{m+1}{m}|A_{k}|\\
    & \leq \big(1-\frac{\mu}{2}\big)|A_{k}|.
\end{align*}
Thus, for $k\geq k_2$, we have
\begin{align*}
     |A_k|  \leq \bigg(1-\frac{\mu}{2}\bigg)^{k-k_2}|A_{k_2}|.
\end{align*}
In particular, for any natural number $k$, we have
\begin{align*}
    \bigg|\bigg\{(x,t): \tilde{u}(x,t) > L^{km}, |x|_{\infty} < c_n \hspace{2mm} \& \hspace{2mm} -1 < t < -1+\frac{c_n^2}{2} \bigg\}\bigg| &\leq \bigg(1-\frac{\mu}{2}\bigg)^{k-k_2}|A_{k_2}|\\
    &\leq C_1\bigg(1-\frac{\mu}{2}\bigg)^{k},
\end{align*}
where $C_1$ is a universal constant.\\
This implies 
\begin{align*}
    |\{\tilde{u} >\tau\} \cap \hat{K}| \leq C_1\tau^{-{\epsilon}_1},
\end{align*}
where ${\epsilon}_1=-\frac{\text{ln}(1-\mu/2)}{m (\text{ln} L)}$.\\
Now, for  ${\epsilon}=-\frac{\text{ln}(1-\mu/2)}{2m (\text{ln} L)}$ we have
\begin{align*}
     \int_{\hat{K}}\tilde{u}^{\epsilon}dxdt&=\epsilon\int_0^{\infty}\tau^{\epsilon-1}|\{\tilde{u} >\tau\} \cap \hat{K}|d\tau\\
     &\leq \epsilon|\hat{K}|\int_0^{1}\tau^{\epsilon-1}d\tau+\epsilon\int_1^{\infty}\tau^{\epsilon-1}|\{\tilde{u} >\tau\} \cap \hat{K}|d\tau\\
     &\leq \epsilon|\hat{K}|\int_0^{1}\tau^{\epsilon-1}d\tau+\epsilon C_1\int_1^{\infty}\tau^{\epsilon-1}\tau^{-{\epsilon}_1}d\tau\\
     &\leq \epsilon|\hat{K}|\int_0^{1}\tau^{\epsilon-1}d\tau+\epsilon C_1\int_1^{\infty}\tau^{\frac{-{\epsilon}_1}{2}-1}d\tau\\
     &\leq C_2,
\end{align*}
where $C_2$ is a universal constant.\\
Thus, there exist universal constants $C$ (=max$\{C_1,C_2\}$) and $\epsilon >0$ such that 
\begin{align*}
    |\{\tilde{u} >\tau\} \cap \hat{K}| \leq C\tau^{-{\epsilon}}
\end{align*}
and
\begin{align*}
     \bigg(\int_{\hat{K}}\tilde{u}^{\epsilon}\bigg)^{\frac{1}{\epsilon}} \leq C.
 \end{align*}
\end{proof}

\begin{lemma}\label{super}
Let $u \in C(Q_2)$ be a solution of (\ref{meq1}) and (\ref{meq2}) with $u(0,0)=1$. Then, there exist universal constants $L_0$ and $\sigma$ such that for $\nu=\frac{L_0}{L_0-1/2}$, the following holds:\\
If \[(x_0,t_0) \in \{(x,t): |x|_{\infty}<c_n/2, -1+c_n^2/4< t < -1+c_n^2/2\}\] and $l$ is a natural number such that 
\begin{align*}
    \tilde{u}(x_0,t_0)\geq \nu^lL_0,
\end{align*}
where $\tilde{u}(x,t)=u(a_{k_1}x,a_{k_1}^2t)$.
Then, 
\begin{align*}
   \text{sup}_{Q_{r_l}(x_0,t_0)} \tilde{u}(x,t)>\nu^{l+1}L_0,
\end{align*}
where $r_l=\sigma \nu^{-(l+1)\frac{\epsilon}{n+2}}\big(\frac{L_0}{2}\big)^{-\frac{\epsilon}{n+2}} < 1$. Here, $\epsilon >0$ is a universal constant in Lemma \ref{wh}.
\end{lemma}
\begin{proof}
Choose $\sigma$  such that $\sigma^{n+2} > \frac{C}{\mu c_n^{n+2}}$, where $C$ is a constant as in Lemma \ref{wh}.

Assume, on the contrary, $\text{sup}_{Q_{r_l}}\tilde{u}(x,t)\leq \nu^{l+1}L_0$. We will choose the universal constant $L_0$ later. We define
\begin{align*}
    G=\{(x,t):|x-x_0|_{\infty}<r_l c_n,t_0-r_l^2<t<t_0+r_l^2(c_n^2-1))\}.
\end{align*}
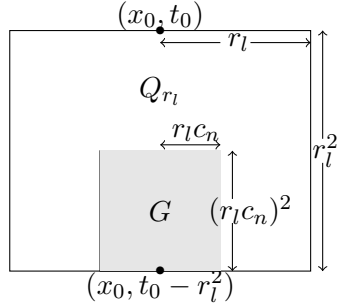
\begin{center}
    \begin{tikzpicture}[scale=0.8];
    \draw (2,-0.25) node{$(x_0,t_0-r_l^2)$};
    \draw (-0.5,0) rectangle (4.5,4);
    \draw (2,3) node{$Q_{r_l}$};
    \draw (1,0) rectangle (3,2);
    \fill[color=gray!20] (1,0) rectangle (3,2);
    \draw (2,1) node{$G$};
    \draw (2,4) node{\tiny{$\bullet$}};
    \draw (2,4.25) node{$(x_0,t_0)$};
   \draw (3.5,1.0) node{$(r_l c_n)^2$};
    \draw[->] (3.2,1.2)--(3.2,2);
    \draw[<-] (3.2,0)--(3.2,.8);
    \draw (2.6,2.3) node{$r_l c_n$};
    \draw[<->] (2,2.1)--(3,2.1);
    \draw[<-] (2,3.8)--(3,3.8);
    \draw[->] (3.5,3.8)--(4.5,3.8);
    \draw (3.3,3.8) node{$r_l$};
    \draw [<-] (4.7,0)--(4.7,1.8);
    \draw [->] (4.7,2.2)--(4.7,4);
    \draw (4.75,2) node{$r_l^2$};
    \draw (2,0) node{\tiny{$\bullet$}};
    \end{tikzpicture}
\end{center}
Now, we will estimate measure of $G$ using $L^{\epsilon}$-estimate and basic measure estimate.\\
First, we will estimate
\begin{align*}
    \bigg|\bigg\{(x,t)\in G:\tilde{u}(x,t)\leq\nu^{l+1}\frac{L_0}{2}\bigg\}\bigg|.
\end{align*}
Define $v:Q_1\rightarrow\R$ as 
\begin{align*}
    v(x,t)&=\frac{\nu}{\nu-1}-\frac{\tilde{u}(r_lx+x_0,r_l^2t+t_0)}{(\nu-1)\nu^lL_0}\\
&=\frac{\nu}{\nu-1}-\frac{u(a_{k_1}r_lx+a_{k_1}x_0,a_{k_1}^2r_l^2t+a_{k_1}^2t_0)}{(\nu-1)\nu^lL_0}
\end{align*}
Now we have to show that $v$ is a supersolution. By Lemma \ref{scaling}, we will be done if
\begin{align*}
    a_{k_1}r_l\leq \frac{1}{L_2\eta((\nu-1)\nu^lL_0)+L_2}.
\end{align*}
Since $a_{k_1}\leq1$, we will be done if
\begin{align*}
    r_l\leq \frac{1}{L_2\eta((\nu-1)\nu^lL_0)+L_2}.
\end{align*}
Since $\eta(t)\geq 1$ for all $t$, we will be done if
\begin{align*}
    2L_2\eta((\nu-1)\nu^lL_0)&\leq r_l^{-1},\\
    \text{i.e.,} \hspace{2mm} 2L_2\eta((\nu-1)\nu^lL_0)&\leq {\sigma}^{-1} \nu^{(l+1)\frac{\epsilon}{n+2}}\big(\frac{L_0}{2}\big)^{\frac{\epsilon}{n+2}}
\end{align*}
Since $\eta(st)\leq {\Lambda}_0\eta(s)\eta(t)$ so we will be done if
\begin{align*}
    2L_2{\Lambda}_0\eta({\nu}^l)\eta((\nu-1)L_0)&\leq {\sigma}^{-1} \nu^{(l+1)\frac{\epsilon}{n+2}}\big(\frac{L_0}{2}\big)^{\frac{\epsilon}{n+2}}\\
    \text{i.e.,}\frac{2L_2{\Lambda}_0\eta((\nu-1)L_0)}{{\nu}^\frac{\epsilon}{n+2}\big(\frac{L_0}{2}\big)^{\frac{\epsilon}{n+2}}}\frac{\eta({\nu}^l)}{\nu^{l\frac{\epsilon}{n+2}}}&\leq {\sigma}^{-1}.
\end{align*}
Now from part (ii) of Lemma \ref{etaprop}, we have $\underset{t\rightarrow\infty}{\text{lim}}\frac{\eta(t)}{t^{\frac{\epsilon}{n+2}}}=0$, which implies that for all $t\geq 1$, we have $\frac{\eta(t)}{t^{\frac{\epsilon}{n+2}}} \leq C_1$ for some universal constant $C_1$.\\
Also, for $L_0 \geq 2$, we have $1<\nu<\frac{3}{2}$ and $\frac{1}{2}<(\nu-1)L_0<\frac{3}{4}$.\\
We will be done if 
\begin{align*}
    \frac{C_2}{\big(\frac{L_0}{2}\big)^{\frac{\epsilon}{n+2}}} \leq {\sigma}^{-1},
\end{align*}
where $C_2$ is universal constant.\\
We will be done if $L_0 \geq \text{max}\{L,2\}$ is  large enough to satisfy above inequality; but $L_0$ is in our hand, so we are done. Also, We will choose $L_0$ large enough to make sure that $r_l < 1$.
Thus $v$ is a supersolution in $Q_1$.
Note that $\tilde{u} \leq \nu^{l+1}L_0$ in $Q_{r_l}$, which implies $v \geq 0.$ Also, $\tilde{u}(x_0,t_0) \geq \nu^l L_0$ implies $v(0,0) \leq 1$.\\
So, by Lemma \ref{measureuniform},
 we have
 \begin{align*}
     |\{v>L\}\cap K_1| &\leq (1-\mu)|K_1|.
 \end{align*}
 Since $L_0>L$, above inequality in particular implies, 
  \begin{align*}
     |\{v\geq L_0\}\cap K_1| &\leq (1-\mu)|K_1|,
 \end{align*}
 which by rescaling and the fact that, $\tilde{u} \leq \nu^{l+1}\frac{L_0}{2}$ iff $v\geq L_0$, gives 
 \begin{align*}
     |\{\tilde{u}\leq \nu^{l+1}\frac{L_0}{2}\} \cap G| \leq (1-\mu)|G|.
 \end{align*}
 Now, by Lemma \ref{wh}
 we have
 \begin{align*}
    \bigg|\bigg\{(x,t)\in G:\tilde{u}(x,t) > \nu^{l+1}\frac{L_0}{2}\bigg\}\bigg| \leq C\bigg(\nu^{l+1}\frac{L_0}{2}\bigg)^{-\epsilon}.
\end{align*}
Hence,
\begin{align*}
    |G| \leq  C\bigg(\nu^{l+1}\frac{L_0}{2}\bigg)^{-\epsilon} +  (1-\mu)|G|,
\end{align*}
which implies 
\begin{align*}
   \mu |G| &\leq  C\bigg(\nu^{l+1}\frac{L_0}{2}\bigg)^{-\epsilon},\\
   \text{i.e.,}\hspace{2mm}\mu\sigma^{n+2}\nu^{-(l+1)\epsilon}\bigg(\frac{L_0}{2}\bigg)^{-\epsilon} c_n^{n+2} &\leq C\bigg(\nu^{l+1}\frac{L_0}{2}\bigg)^{-\epsilon}.
\end{align*}
Then, we have
\begin{align*}
    \sigma^{n+2} \leq \frac{C}{\mu c_n^{n+2}},
\end{align*}
which is a contradiction to $\sigma^{n+2} > \frac{C}{\mu c_n^{n+2}}$.
This proves the lemma.
\end{proof}
\begin{lemma}\label{final}
Let $u \in C(Q_2)$ be a solution of (\ref{meq1}) and (\ref{meq2}) with $u(0,0)=1$. Then, there exists a universal constant $l_0 $ such that 
\begin{align}\label{aaa}
    \text{sup}_A \tilde{u}(x,t) \leq \nu^{l_0}L_0,
\end{align}
where $L_0$ is universal constant from Lemma \ref{super}, $\tilde{u}(x,t)=u(a_{k_1}x,a_{k_1}^2t)$, and\\ $A=\Big\{(x,t):|x|_{\infty} \leq \frac{c_n}{2}, -1+ \frac{c_n^2}{4} \leq t \leq -1+\frac{c_n^2}{2}\Big\}$.
\end{lemma}
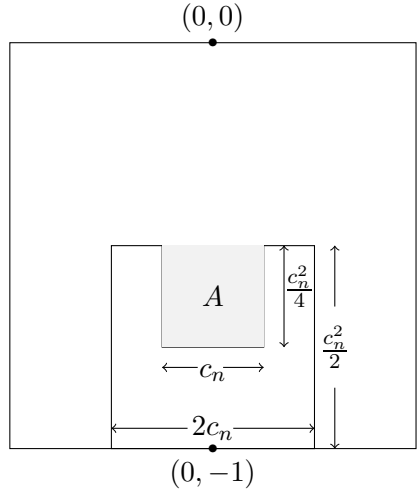
\begin{center}
    \begin{tikzpicture}[scale=1.35]
    \draw (0,0) rectangle (4,4);
    \draw (1,0) rectangle (3,2);
    \draw (1.5,1) rectangle (2.5,2);
    \fill[color=gray!10] (1.5,1) rectangle (2.5,2);
    \draw (2,1.5) node{$A$};
    \draw (2,4) node{\tiny{$\bullet$}};
    \draw (2,4.25) node{$(0,0)$};
    \draw (2,-0.25) node{$(0,-1)$};
    \draw (2,0) node{\tiny{$\bullet$}};
    \draw[<-] (1.5,0.8)--(1.85,0.8);
    \draw[->] (2.1,0.8)--(2.5,0.8);
    \draw (2,0.75) node{$c_n$};
    \draw[<->] (2.7,1)--(2.7,2);
    \draw (2.85,1.6) node{$\frac{c_n^2}{4}$};
    \draw[<-] (1,0.2)--(1.75,0.2);
    \draw[->] (2.2,0.2)--(3,0.2);
    \draw (2,0.2) node{$2c_n$};
    \draw[<-] (3.2,0)--(3.2,0.6);
    \draw[->] (3.2,1.4)--(3.2,2);
    \draw (3.2,1) node{$\frac{c_n^2}{2}$};
    \end{tikzpicture}
\end{center}
\begin{proof}
Choose $l_0$ such that 
\begin{align*}
    \sum_{j=l_0}^{\infty}r_j \leq \frac{c_n}{2} \hspace{3mm}\text{and} \hspace{3mm} \sum_{j=l_0}^{\infty}r_j^2 \leq \frac{c_n^2}{8},
\end{align*}
where $r_j$ is as in Lemma \ref{super}. We can choose such an $l_0$ because $\nu >1$.\\
Assume (\ref{aaa}) is not true. Then, there exists $(x_{l_0},t_{l_0}) \in A$ such that
\begin{align*}
    \tilde{u}(x_{l_0},t_{l_0}) > \nu^{l_0}L_0.
\end{align*}
  By Lemma \ref{super}, there exists $(x_{l_0+1},t_{l_0+1}) \in Q_{r_{l_0}}(x_{l_0},t_{l_0})$ such that
  \begin{align*}
    \tilde{u}(x_{l_0+1},t_{l_0+1}) > \nu^{l_0+1}L_0,
\end{align*}
$|x_{l_0+1}-x_{l_0}|_{\infty} <r_{l_0}$ and $t_{l_0}\geq t_{l_0+1} > t_{l_0}-r_{l_0}^2$.\\
By repeating this process we get a sequence $(x_{l},t_{l})$ for $l >l_0$ such that
\begin{align}\label{12}
    \tilde{u}(x_{l},t_{l}) > \nu^{l}L_0,
\end{align}
$|x_{l+1}-x_{l}|_{\infty} <r_{l}$ and $t_{l} \geq t_{l+1} > t_{l}-r_{l}^2$.\\
Note that, for $l>l_0$,
\begin{align*}
    |x_{l}|_{\infty} &\leq |x_{l_0}|_{\infty} +\sum_{j=l_0}^{l-1}|x_{j+1}-x_{j}|_{\infty}\\
    &\leq \frac{c_n}{2} +\sum_{j=l_0}^{\infty}r_j\\ &\leq \frac{c_n}{2} + \frac{c_n}{2}\\
    &=c_n.
\end{align*}
Also,
\begin{align*}
    t_{l_0} \geq t_l &\geq t_{l_0} -\sum_{j=l_0}^{l-1}r_j^2\\ &\geq t_{l_0} -\sum_{j=l_0}^{l-1}r_j^2\\ &\geq -1 + \frac{c_n^2}{4} -\frac{c_n^2}{8}\\ &\geq -1+\frac{c_n^2}{8}.
\end{align*}
Thus, $(x_{l},t_{l})$ is a bounded sequence in $K_1$. Now, by continuity of $u$ in $Q_2$, $\tilde{u}(x_{l},t_{l})$ is a bounded sequence, which is a contradiction to (\ref{12}). This proves the lemma.
\end{proof}

Now, we will prove our main theorem Theorem \ref{mthm}, which is a direct consequence of the previous lemma.
\begin{proof}[Proof of Theorem \ref{mthm}]
Let $u \in C(Q_2)$ be a solution of (\ref{eq1}) and (\ref{eq2}).
Define
\begin{align*}
    v(x,t):=\frac{u(a_0x,a_0^2t)}{u(0,0)}.
\end{align*}
Then, by Lemma \ref{scaling}, $v$ is a solution of (\ref{eq1}) and (\ref{eq2}).\\
Now, define 
\begin{align*}
    v_1(x,t):=v(r_1x,r_1^2t).
\end{align*}
Then, $v_1$ is a solution of (\ref{meq1}) and (\ref{meq2}) with $v_1(0,0)=1$.\\
Then, by Lemma \ref{final}, we get
\begin{align*}
    \text{sup}_A \tilde{v_1}(x,t) \leq C,
\end{align*}
where $C=\nu^{l_0}L_0$ is a universal constant,  $\tilde{v_1}(x,t)=v_1(a_{k_1}x,a_{k_1}^2t)$, and\\ $A=\Big\{(x,t):|x|_{\infty} \leq \frac{c_n}{2}, -1+ \frac{c_n^2}{4} \leq t \leq -1+\frac{c_n^2}{2}\Big\}$.\\
Hence, we have
\begin{align*}
    \text{sup}_A u(a_0 a_{k_1} r_1 x,(a_0 a_{k_1} r_1)^2t) \leq C u(0,0),
\end{align*}
i.e., 
\begin{align*}
    \text{sup}_A u(a_0 \gamma x,(a_0 \gamma)^2 t) \leq C u(0,0),
\end{align*}
where $\gamma=a_{k_1} r_1$.
This proves the result.
\end{proof}

\section{Appendix}\label{a}


\begin{proof}[Proof of Lemma \ref{barrier}]
 We will work with the same barrier function which is constructed in \cite{is}. In our case, we have to take care of nonlinearity as well. So, here we will do calculations more carefully. Our idea of scheme is to make $P^+(D^2h)-h_t$ strictly negative instead of nonpositive as in \cite{is}. With the help of this  negative quantity and favourable scaling we will take care of nonlinearity.\\ First we will define $h$ in the region
\begin{align*}
  D= \bigg\{(x,t):|x|<1 \hspace{2mm}\&\hspace{2mm} \frac{c_n^2}{36n}\leq t \leq \frac{1}{36n} \bigg\} 
\end{align*}
as
\begin{align*}
   h(x,t)=t^{-p}H\Big(\frac{x}{\sqrt{t}}\Big),
\end{align*}
where $H:\R^n\rightarrow\R$ is defined as
\[H(y)=\begin{cases}
    \text{some smooth and bounded function between $-((6\sqrt{n})^q(2^q-1)^{-1})$ and $-1$} &\quad \text{if}\hspace{1.5mm} |y|\leq3\sqrt{n},\\
    ((6\sqrt{n})^{-q}-|y|^{-q}) &\quad \text{if}\hspace{1.5mm}3\sqrt{n}\leq|y|\leq6\sqrt{n},\\
    0 &\quad \text{if}\hspace{1.5mm}|y|\geq6\sqrt{n}.
    \end{cases}
    \]
To calculate $P^+(D^2h)-h_t$, we will do some calculations:
\begin{align*}
    h_i&=t^{-p}H_i\bigg(\frac{x}{\sqrt{t}}\bigg);
     \hspace{5mm} h_{ij}=t^{-p}H_{ij}\bigg(\frac{x}{\sqrt{t}}\bigg)\frac{1}{t};\\
    h_t&=-pt^{-p-1}H\bigg(\frac{x}{\sqrt{t}}\bigg)+t^{-p}\nabla H\bigg(\frac{x}{\sqrt{t}}\bigg)\cdot\frac{x}{t^{-3/2}}\bigg(-\frac{1}{2}\bigg)\\
    &=-p t^{-p-1}H\bigg(\frac{x}{\sqrt{t}}\bigg)-\frac{1}{2}t^{-p-1}\nabla H\bigg(\frac{x}{\sqrt{t}}\bigg)\cdot\frac{x}{\sqrt{t}}.
\end{align*}
Here, $h_i$ represents the partial derivative of $h$ with respect to $x_i$.\\
Thus, we have
\begin{align}\label{H}
    P^+(D^2h)-h_t=t^{-p-1}P^+\bigg(D^2H\bigg(\frac{x}{\sqrt{t}}\bigg)\bigg)+p t^{-p-1}H\bigg(\frac{x}{\sqrt{t}}\bigg)+\frac{1}{2}t^{-p-1}\nabla H\bigg(\frac{x}{\sqrt{t}}\bigg)\cdot\frac{x}{\sqrt{t}}.
\end{align}
We will calculate the derivative of $H$ in the region $ \{y\in \R^n:3\sqrt{n} \leq |y| \leq 6\sqrt{n}\}$ as
\begin{align*}
    H_{i}(y)&=q|y|^{-q-1}\frac{y_i}{|y|}=q|y|^{-q-2}{y_i},\\
    H_{ij}(y)&=-q(q+2)|y|^{-q-4}{y_i}{y_j}+q|y|^{-q-2} \delta_{ij}.
\end{align*}
Hence, in the region $ \{y \in \R^n:3\sqrt{n} \leq |y| \leq 6\sqrt{n}\}$, eigenvalues of $D^2H(y)$ are $q|y|^{-q-2}$ with multiplicity $n-1$ and $-q(q+1)|y|^{-q-2}$ with multiplicity 1 so,
\begin{align*}
    P^+(D^2H(y))&=q(\Lambda(n-1)-\lambda(q+1))|y|^{-q-2}.
\end{align*}
Then, we have 
\begin{align}\label{star}
     P^+(D^2h)=q(\Lambda(n-1)-\lambda(q+1))t^{-p-1}\bigg(\frac{|x|}{\sqrt{t}}\bigg)^{-q-2}.
\end{align}
Note that
 \begin{align*}
   \nabla H\bigg(\frac{x}{\sqrt{t}}\bigg)\cdot\frac{x}{\sqrt{t}}&=q \bigg(\frac{|x|}{\sqrt{t}}\bigg)^{-q}.
\end{align*}
Thus, in the region $\Big\{(x,t)\in D:3\sqrt{n}\leq\frac{|x|}{\sqrt{t}}\leq 6\sqrt{n}\Big\}$, we have 
\begin{align*}
    P^+(D^2h)-h_t\leq q(\Lambda(n-1)-\lambda(q+1))t^{-p-1}\bigg(\frac{|x|}{\sqrt{t}}\bigg)^{-q-2}+p t^{-p-1}H\bigg(\frac{|x|}{\sqrt{t}}\bigg)+\frac{q}{2}t^{-p-1} \bigg(\frac{|x|}{\sqrt{t}}\bigg)^{-q}.
\end{align*}
By definition  $H(y)\leq 0$ for $y \in \R^n$. So, we have
\begin{align*}
    P^+(D^2h)-h_t\leq q(\Lambda(n-1)-\lambda(q+1))t^{-p-1}\bigg(\frac{|x|}{\sqrt{t}}\bigg)^{-q-2}+\frac{q}{2}t^{-p-1} \bigg(\frac{|x|}{\sqrt{t}}\bigg)^{-q}.
\end{align*}
Using $\frac{|x|}{\sqrt{t}} \le 6\sqrt{n}$, we get
\begin{align*}
    P^+(D^2h)-h_t\leq q(\Lambda(n-1)-\lambda(q+1) +18n)t^{-p-1}\bigg(\frac{|x|}{\sqrt{t}}\bigg)^{-q-2}.
\end{align*}
Choose $q$ large enough such that $\Lambda(n-1)-\lambda(q+1) +18n \leq -1$. Then, we have
\begin{align}\label{112}
    P^+(D^2h)-h_t\leq -qt^{-p-1}\bigg(\frac{|x|}{\sqrt{t}}\bigg)^{-q-2}
\end{align}
in $\Big\{(x,t)\in D:3\sqrt{n}\leq\frac{|x|}{\sqrt{t}}\leq 6\sqrt{n}\Big\}$.\\
\\
Now we will estimate $ P^+(D^2h)-h_t$ in the region $\Big\{(x,t)\in D:\frac{|x|}{\sqrt{t}}\leq 3\sqrt{n}\Big\}$.\\ 
Note that $H(y)$ is smooth in $|y|\leq3\sqrt{n}$, so there exists $M$ such that 
\begin{align*}
    P^+\Big(D^2H\Big(\frac{x}{\sqrt{t}}\Big)\Big)+\frac{1}{2}\nabla H\Big(\frac{x}{\sqrt{t}}\Big)\cdot\frac{x}{\sqrt{t}} \leq M.
\end{align*}
 Choose large $p$ such that $-p+M\leq-1$. Then, from (\ref{H}) and $H(y) \le -1$ in $|y| \le 3 \sqrt{n}$, we have 
 \begin{align*}
      P^+(D^2h)-h_t\leq -t^{-p-1} \hspace{2mm} \text{in} \hspace{2mm} \Big\{(x,t)\in D:\frac{|x|}{\sqrt{t}}\leq 3\sqrt{n}\Big\}. 
 \end{align*}
 Thus, in $D \cap \{(x,t):h(x,t)\not=0\}$, we get
 \begin{align}\label{113}
      P^+(D^2h)-h_t\leq -t^{-p-1}. 
 \end{align}
Now we will extend the definition of $h$ in $\{(x,t):|x|\leq 1,1/(36n) \leq t \leq 1\} $  as
\begin{align*}
    h(x,t)=e^{(\kappa (t-{1/{36n}}))}h(x,1/36n),
\end{align*}
where $\kappa$ is chosen such that 
$\underset{\{|x|<1\}}{\text{inf}}\frac{P^+(D^2h(x,1/36n))}{h(x,1/36n)}\geq \kappa+1$.
(We can choose $\kappa <0$, i.e., inf exists because 
\begin{itemize}
    \item For $|x| \leq \frac{1}{2}$, $h(x,1/36n) \leq -(36n)^p$ and $h$ is smooth so $P^+(D^2h(x,1/36n))$ is bounded. Hence, $\frac{P^+(D^2h(x,1/36n))}{h(x,1/36n)}$ is bounded below.
    \item For $1>|x| \geq \frac{1}{2}$, from (\ref{star}), we have $P^+(D^2h(x,1/36n)) < 0$ and by definition $h < 0$ so $\frac{P^+(D^2h(x,1/36n))}{h(x,1/36n)}$ is bounded below by $0$.)
\end{itemize}
Now, in $\{(x,t):|x|\leq 1,1/(36n) \leq t \leq 1\} $, 
\begin{align*}
    P^+(D^2h)-h_t=e^{(\kappa (t-{1/{36n}}))}(P^+(D^2h)-\kappa h(x,1/36n)).
\end{align*}
Take $m >0$ such that $-m=\underset{\{{1/2}<|x|<1\}}{\text{min}}P^+(D^2h(x,1/36n))$. Note that $h(x,1/36n)=0$ for $|x|=1$. Hence, by continuity of $h$, we can choose $\delta >0$ such that  \[-m-\kappa h(x,1/36n) \leq \frac{-m}{2}\] in $\{1-\delta <|x| \le 1\}$,
which gives us 
\begin{align}
    P^+(D^2h)-h_t\leq \frac{-m}{2}e^{(\kappa (t-{1/{36n}}))} \hspace{2mm} \text{in} \hspace{2mm} \{(x,t):1-\delta < |x| \le 1,1/(36n) \leq t \leq 1\}.
\end{align}
Now, using the definition of $\kappa$, we have
\begin{align*}
    P^+(D^2h)-h_t&=e^{(\kappa (t-{1/{36n}}))}(P^+(D^2h)-\kappa h(x,1/36n))\\
    &\leq e^{\kappa (t-{1/{36n}})} h(x,1/36n).
\end{align*}
Thus, in $\{(x,t):1-\delta < |x| \le 1,1/(36n) \leq t \leq 1\}$, we have 
\begin{align*}
    P^+(D^2h)-h_t \leq e^{\kappa (t-{1/{36n}})} \underset{\{|x| \leq 1-\delta\}}{\text{max}}h(x,1/36n).
\end{align*}

Thus, in $\{(x,t):|x|\leq 1,1/(36n) \leq t \leq 1\} $, we have
\begin{align}\label{114}
    P^+(D^2h)-h_t &\leq \text{max}\{-m/2,\underset{\{|x| \leq 1-\delta\}}{\text{max}}h(x,1/36n)\}e^{\kappa (t-{1/{36n}})}\\
    &\leq \text{max}\{-m/2,\underset{\{|x| \leq 1-\delta\}}{\text{max}}h(x,1/36n)\}\label{114}.
\end{align}
Hence, from (\ref{112}),(\ref{113}) and (\ref{114}) we get 
\begin{align*}
    P^+(D^2h)-h_t&\leq \text{max}\{-m/2,\underset{\{|x| \leq 1-\delta\}}{\text{max}}h(x,1/36n),-(36n)^{p+1},-q(6\sqrt{n})^{q+2}c_n^{-2p-2}(36\sqrt{n})^{p+1}\}\\
    &=:-c.
\end{align*}
Note that $h(x,1/36n) <0$ in $|x| \leq 1-\delta$, implies $\underset{\{|x| \leq 1-\delta\}}{\text{max}}h(x,1/36n) <0$. Hence, $c>0$ depends only on dimension and ellipticity constants.
We will define the following subsets of $Q_1(0,1).$
\begin{align*}
    \hat{K_1}&:=(-c_n,c_n)^n \times (0,c_n^2),\\
    \hat{K_3}&:=(-3c_n,3c_n)^n \times (c_n^2,1),
\end{align*}
where $c_n=(10n)^{-1}.$
Note that the way we constructed $h$, it vanishes only in $\{(x,t):|x| \geq 6\sqrt{nt},t<(36n)^{-1}\}$, and $\hat{K}_3$ does not intersects with it, so
\begin{align}\label{al}
	\alpha:=\text{sup}_{\hat{K}_3}h(x,t)<0.
\end{align} 
Now we will define function $\tilde{h}:Q_1(0,1) \rightarrow \R$ as

\[\tilde{h}(x,t)=\begin{cases}
          \frac{-2h(x,t)}{\text{sup}_{\hat{K}_3}h} &\quad \text{whenever $h$ is defined,}\\
          0 &\quad \text{otherwise}.
          \end{cases}
         \]

Take a smooth function $\bar{h}$ defined in $\hat{K}_1$ such that $\bar{h} = \tilde{h}$ on boundary of $\hat{K}_1$ (set theoretic boundary).
Define $h:Q_1 \rightarrow \R$ (this $h$ is not same as the function we were defining earlier) as
\[h(x,t)=\begin{cases}
          \tilde{h}(x,t+1) &\quad \text{if} (x,t)\in Q_1\setminus {K}_1\\
          \bar{h}(x,t+1)  &\quad \text{if} (x,t)\in {K}_1.
          \end{cases}
         \]
Note that $h$ is a Lipschitz function so $\phi(|D h|)$ is bounded, say by $M_1$.\\
Now, choose $r_0>0$ such that 
\begin{align}
    \frac{2c}{\alpha}+\Lambda_0 r_0 \eta \bigg(\frac{1}{r_0}\bigg) M_1 \leq 0,
\end{align}
where $\alpha$ is define in (\ref{al}). This will imply for $r \leq r_0$ that, 
\begin{align}\label{121}
    \frac{2c}{\alpha}+\Lambda_0 r \eta \bigg(\frac{1}{r}\bigg) M_1 \leq 0.
\end{align}
 
Then, we have (in viscosity sense)
\begin{align*}
    P^+(D^2h)-h_t+\tilde{\phi}(|D h|)\leq g,
\end{align*}
and support of $g$ lies in ${K}_1$. 
Hence, $h$ is required barrier function.
\end{proof}
\noindent
\begin{proof}
[Proof of Lemma \ref{measureuniform}]
 By approximating with inf-convolution
\begin{align*}
    u_{\epsilon}(x,t)=\underset{(y,s) \in Q_2}{\text{inf}}\bigg\{u(y,s)+\frac{1}{\epsilon}(|x-y|^2+|t-s|^2)\bigg\}
\end{align*}
We may assume $u$ is semiconcave in $x$ variable and Lipschitz in $x$ and $t$ variable.\\
Define $v(x,t)=u(x,t)+h(x,t)$ in $Q_1$, where $h$ is the barrier function constructed in Lemma \ref{barrier}.
We will choose $r_1$ later such that $r_1 \leq r_0$.\\
Now, using properties of Pucci operators, increasing nature of $\phi$, $r_1 \le r_0$, and properties of barrier function, we obtain $v$ is a viscosity supersolution of  
\begin{align}\label{100}
    \psr(D^2v)-v_t \leq \tilde{\phi}(|Dv|+|Dh|)-\tilde{\phi}(|Dh|)+g.
    \end{align}
Now, extend $v$ by zero in $Q_2$, and denote the convex envelope of $\text{min}\{v,0\}$ by $\Gamma_v$.\\
Using (\ref{100}), we get (in viscosity sense)
\begin{align*}
    \psr(D^2v)-v_t \leq M
\end{align*}
for some $M$.
Then,  $\Gamma_v$ is $C^{1,1}$ with respect to $x$ and Lipschitz with respect to $t$ in $\{v=\Gamma_v\}$ ( see Corollary $3.17$ in \cite{W1}).
Since $v$ is a viscosity supersolution of equation (\ref{100}) and $\Gamma_v \leq v$ we have 
\begin{align}\label{eqgammav}
    P^-(D^2(\Gamma_v))-(\Gamma_v)_t \leq \tilde{\phi}(|D\Gamma_v|+|Dh|)- \tilde{\phi}(|Dh|) +g \hspace{2mm} \text{a.e} \hspace{2mm} \text{in} \hspace{2mm} \{v=\Gamma_v\}.
\end{align}
Define $E:=\{v=\Gamma_v\} \cap \{(x,t) \in Q_1: |D\Gamma_v(x,t)| \leq 1\}$.

\textbf{Claim:} There exists a constant $b,$ depending on $\phi$ and $n$ such that for a.e. in $E$ the following holds:
\begin{align}\label{eqfinal}
    P^-(D^2(\Gamma_v))-(\Gamma_v)_t \leq b |D\Gamma_v|+ \tilde{\phi}(2) +g.
\end{align}
Proof of claim: We will prove this in two cases depending on the value of $|Dh|$ on that point.\\
Case 1:  We assume  $(x,t) \in E$ such that $|Dh|(x,t)\leq 1$. In this case, we can bound $\tilde{\phi}(|D\Gamma_v|+|Dh|)- \tilde{\phi}(|Dh|)$ by $\tilde{\phi}(2)$ because $\phi\geq0$ and we have $|D\Gamma_v|\leq 1$ in $E$.\\
Case 2: We assume $(x,t) \in E$ such that $|Dh(x,t)|\geq 1$.
Since $h$ is Lipschitz, $\text{sup}_{Q_1}$$|Dh|$ exists and it depends only on $n$ owing to the way we constructed $h$. Since $\phi$ is locally Lipschitz, we have
\begin{align}
   \phi(|D\Gamma_v|+|Dh|)- \phi(|Dh|) \leq b |D\Gamma_v|,
\end{align}
Where
\begin{align}\label{defb}
    b= \text{max}\{|D\phi(y)|:1 \leq y \leq \text{sup}_{Q_1}|Dh|+1\}.
\end{align}
Then, we have
\begin{align}
   \tilde{\phi}(|D\Gamma_v|+|Dh|)- \tilde{\phi}(|Dh|) \leq b |D\Gamma_v|.
\end{align}
From both the cases and (\ref{eqgammav}) we get our claim.\\
Recall for given $f$, $Gf(x,t)=(Df(x,t), f(x,t)-x\cdot Df(x,t))$.
Since $\Gamma_v$ is $C^{1,1}$ with respect to $x$  and Lipschitz with respect to $t$ in $\{v=\Gamma_v\}$, by Lemma \ref{glem},  $G{\Gamma}_v$ is Lipschitz continuous in $E$. Now,
 Coarea formula and Lemma \ref{glem} gives
 
\begin{align*}
    \int_{G\Gamma_v(E)}\frac{d\xi dh}{|\xi|^{n+1} + \delta} 
        &\leq \int_{E}\frac{|\text{det}(G\Gamma_v)|}{|D\Gamma_v|^{n+1} + \delta}dxdt\\
   &=\int_{E}\frac{|\partial_t\Gamma_v|\hspace{1mm} |\text{det}D^2(\Gamma_v)|}{|D\Gamma_v|^{n+1} + \delta}dxdt,
\end{align*}
where $\delta >0$ is a small universal constant which will be chosen later.
Now, by definition of Monotone envelope (\ref{medef}), we have $\Gamma_v$ is convex with respect to $x$ variable so $\text{det}D^2(\Gamma_v)\geq0$. Also, $\Gamma_v$ is non-increasing with respect to $t$ so $\partial_t\Gamma_v\leq0$. Hence, we get
\begin{align*}
\int_{G\Gamma_v(E)}\frac{d\xi dh}{|\xi|^{n+1} + \delta}
    &\leq \int_{E}\frac{-\partial_t\Gamma_v\hspace{1mm} \text{det}D^2(\Gamma_v)}{|D\Gamma_v|^{n+1} + \delta}dxdt\\
    &=\frac{\lambda^n}{\lambda^n} \int_{E}\frac{-\partial_t\Gamma_v\hspace{1mm} \text{det}D^2(\Gamma_v)}{|D\Gamma_v|^{n+1} + \delta}dxdt\\
     &\leq\frac{\lambda^k\Lambda^{n-k}}{\lambda^n} \int_{E}\frac{-\partial_t\Gamma_v\hspace{1mm} \text{det}D^2(\Gamma_v)}{|D\Gamma_v|^{n+1} + \delta}dxdt,
\end{align*}
where $k$ is number of positive eigenvalues of $D^2(\Gamma_v)$.\\
Now, write determinant as the product of eigenvalues, club positive eigenvalues with $\lambda$ and negative eigenvalues with $\Lambda$, and use AM-GM inequality to get
\begin{align*}
    \int_{G\Gamma_v(E)}\frac{d\xi dh}{|\xi|^{n+1} + \delta}
    &\leq\frac{n+1}{\lambda^n} \int_{E}\frac{(-\partial_t\Gamma_v\hspace{1mm} + P^-(D^2\Gamma_v))^{n+1}}{|D\Gamma_v|^{n+1} + \delta}dxdt.
\end{align*}
Now, using (\ref{eqfinal}), we get
\begin{align*}
    \int_{G\Gamma_v(E)}\frac{d\xi dh}{|\xi|^{n+1} + \delta}
    &\leq\frac{n+1}{\lambda^n} \int_{E}\frac{( b |D\Gamma_v|+ \tilde{\phi}(2) +g)^{n+1}}{|D\Gamma_v|^{n+1} + \delta}dxdt\\
    &\leq C b^{n+1} + C \frac{(\tilde{\phi}(2))^{n+1}}{\delta} + C\frac{|K_1\cap E|}{\delta},
\end{align*}
where $C$ depends on $n$ and ellipticity constant.\\
Now, $\text{inf}_{K_3}u\leq 1$ and $h \le -2$ in $K_3$ imply
\begin{align*}
    \text{inf}_{K_3} v \leq -1,
\end{align*}
which is the same as saying
\begin{align*}
    \text{sup}_{K_3} (\text{min}\{v,0\})^-\geq 1\\
    \implies \text{sup}_{Q_1} (\text{min}\{v,0\})^-\geq 1.
\end{align*}
Also, $u \geq 0$ and $h=0$ on $\partial_p Q_1$ imply $v \geq 0$ on $\partial_p Q_1$.
Hence, by Lemma \ref{setmeas}, we get 
 \begin{align*}
     \Big\{(\xi,h) \in \R^n \times \R:|\xi| \leq \frac{1}{4}, \frac{5}{8} \leq -h \leq \frac{6}{8}\Big\} \subset {G\Gamma_v(E)}.
 \end{align*}
Thus, we have
 \begin{align}\label{eqcon}
   \frac{1}{8} \int_{|\xi|<1/4}\frac{d\xi dh}{|\xi|^{n+1} + \delta}
   \leq C b^{n+1} + C \frac{(\tilde{\phi}(2))^{n+1}}{\delta} + C\frac{|K_1\cap E|}{\delta}.
\end{align}
Since $b$ and $C$ are universal constants, and
\begin{equation*}
\int_{|\xi|<1/4}\frac{d\xi}{|\xi|^{n+1} + \delta} \rightarrow\infty \hspace{2mm}\text{as} \hspace{2mm} \delta \rightarrow 0,
\end{equation*}
we can choose a universal constant $\delta >0$ such that
\begin{align}
     \frac{1}{8} \int_{|\xi|<1/4}\frac{d\xi dh}{|\xi|^{n+1} + \delta}
   \geq C b^{n+1} +2.
\end{align}
Chosen $\delta>0$, choose $r_1 \leq r_0$ universal such that
\begin{align*}
    C(r_1\eta(1/r_1)\phi(2))^{n+1}\leq \delta, \hspace{2mm}\text{which is the same as}\hspace{2mm} C\frac{(\tilde{\phi}(2))^{n+1}}{\delta}\leq1.
\end{align*}
Hence (\ref{eqcon}) becomes
\begin{align*}
     1\leq  C\frac{|K_1\cap E|}{\delta}.
\end{align*}
Take $\mu_1=\frac{\delta}{C}$, which is a universal constant. By definition of $E$, we have
\begin{align*}
    \mu_1 &\leq |K_1\cap E|
    \leq |K_1 \cap \{v=\Gamma_v\}|.
\end{align*}
Using $\Gamma_v$ is monotone envelope of $\text{min}\{v,0\}$ we get,
\begin{align*}
    \mu_1 
    &\leq |K_1 \cap \{v\leq 0\}|.
\end{align*}
Substitute $v=u+h$ to get
\begin{align*}
    |K_1 \cap \{u\leq -h\}| \geq \mu_1.
\end{align*}
Take $L=\text{max}_{Q_1}(-h)$. Then, we have
\begin{align*}
    |K_1 \cap \{u\leq L\}| \geq \mu_1,
\end{align*}
which implies,
\begin{align*}
    |\{(x,t) \in K_1 : u(x,t) \leq L\}| \geq |K_1|-\mu_1.
\end{align*}
Take $\mu=\frac{|K_1|-\mu_1}{|K_1|}$ and get conclusion.
\end{proof}

\end{document}